\documentclass[reqno,12pt]{amsart}
\usepackage{subfigure}
\usepackage{ragged2e}
\usepackage{appendix}
\usepackage{amssymb}
\usepackage{amsfonts}
\usepackage{amsmath}
\usepackage{mathrsfs}
\usepackage{diagbox}
\usepackage{empheq}
\usepackage{color,soul}
\usepackage{bbm}
\usepackage{amsthm, graphicx,empheq}
\usepackage{framed}
\usepackage{booktabs}
\usepackage{tikz,float}
\usepackage[backref,colorlinks,linkcolor=red,anchorcolor=green,citecolor=blue]{hyperref}
\usepackage{varwidth} % 引入 varwidth 包
\usepackage{bm}
\usepackage{lineno}  % show line number for draft. disable for final version
\linenumbers

\newtheorem{lemma}{Lemma}[section]
\newtheorem{theorem}{Theorem}[section]

\newtheorem{remark}{Remark}[section]
\newtheorem{proposition}{Proposition}[section]

\numberwithin{equation}{section}

\newcommand{\rmD}{\mathrm{D}}

\makeatletter
\@namedef{subjclassname@20xx}{\textup{xxxx} Mathematics Subject Classification}
\makeatother

\linenumbers
\nolinenumbers
\begin{document}
	
\title[supersonic flow past a conical wing]{Supersonic flow of a Chaplygin gas past a conical wing with $\Lambda$-shaped cross sections}

%\author{Myoungjean Bae}	
\author{Minghong Han}
\author{Bingsong Long}
\author{Hairong Yuan}

%\address[Myoungjean Bae]
%{Department of Mathematical Sciences, KAIST, 291 Daehak‑Ro, Yuseong‑Gu, Daejeon 34141, Republic of Korea}
%\email{\tt  mjbae@kaist.ac.kr}

\address[Minghong Han]{Center for Partial Differential Equations, School of Mathematical Sciences, East China Normal University, Shanghai 200241, China}\email{\tt 52275500054@stu.ecnu.edu.cn}

\address[Bingsong Long]{School of Mathematics and Statistics, Huanggang Normal University, Hubei 438000, China}\email{\tt longbingsong@hgnu.edu.cn}

\address[H. Yuan]{School of Mathematical Sciences,  Key Laboratory of Mathematics and Engineering Applications (Ministry of Education) \& Shanghai Key Laboratory of PMMP,  East China Normal University, Shanghai 200241, China}\email{\tt hryuan@math.ecnu.edu.cn}

	
	%msc2020

\date{\today}
		
\begin{abstract}
 In this paper, by considering the anhedral angle, we for the first time study the problem of supersonic flow of a Chaplygin gas over a conical wing with $\Lambda$-shaped cross sections, where the flow is governed by the three-dimensional steady isentropic irrotational compressible Euler equations. This work is motivated by the design of the Nonweiler wing, which is one of the simplest waveriders. Mathematically, the problem reduces to a boundary value problem for a nonlinear mixed-type equation in conical coordinates. By introducing a viscosity parameter to treat the degenerate boundary, we use the continuity method to establish the existence of a piecewise smooth self-similar solution to the problem, in the case that the shock is attached to the leading edge of the conical wing. Our results verify part of K\"uchemann's speculation on the conical flow field structures of this type, and also find a new conical flow field structure.
\end{abstract}

\allowbreak
\allowdisplaybreaks
	
\maketitle
	
\tableofcontents %disable for short paper
	
\section{Introduction and main results}\label{sec1}
To improve the lift-to-drag ratio for hypersonic flight vehicles, Nonweiler firstly proposed the concept of a waverider configuration in 1959, and introduced a waverider known as the Nonweiler or caret wing \cite{Nonweiler59}. The Nonweiler wing is a specific class of conical wings with a $\Lambda$-shaped cross section, featuring a planar attached shock wave when it is placed in a supersonic flow; see Figure \ref{fig1} redrawn from \cite[p.912]{Anderson}. In this paper, for the first time, we mathematically justify the global flow field structure of the Nonweiler wing by studying the three-dimensional steady compressible Euler equations of isentropic irrotational  Chaplygin gas flow.  

%-----------------------------------fig-----------
\begin{figure}[htb]
\centering
\includegraphics[scale=0.9]{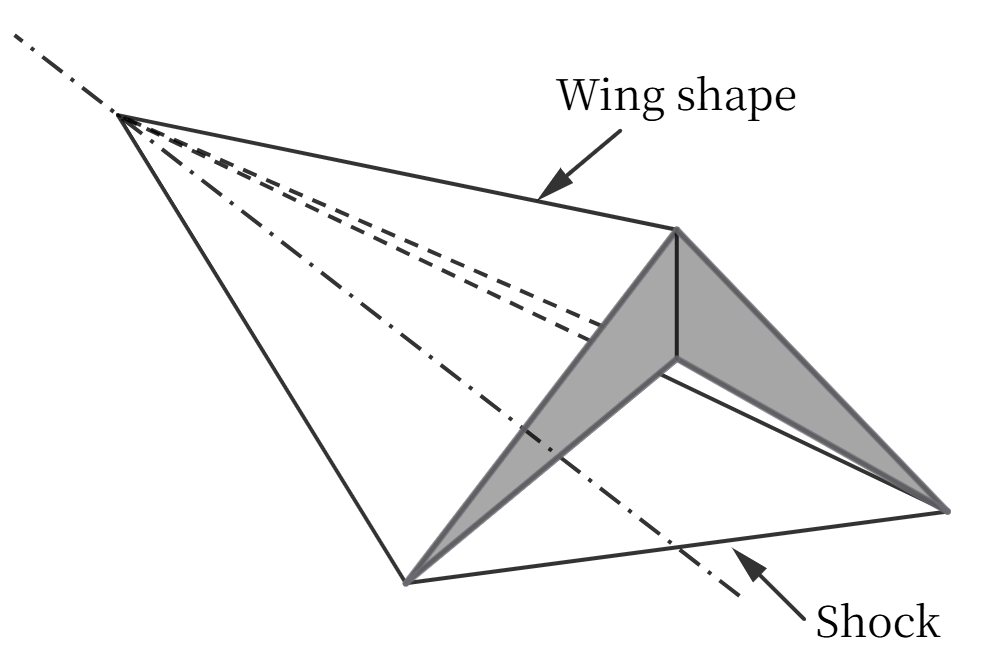}
\caption{Nonweiler or caret wing.}\label{fig1}
\end{figure}
 %-----------------------------------fig-----------
 
Contrary to the traditional aircraft design, where the flow fields and locations of the shocks are computed from given geometric configurations of the aircraft surface and the state of the incoming flow, waverider design involves inversely determining the geometric shape of the compression surface and the flow fields using a known supersonic incoming flow and the locations of shocks. The experimental and computational studies on the waveriders can be found in \cite{Anderson,KD65,Rasmussen80}. Mathematically, inspired by the design of waveriders, Li-Zhang \cite{LZ22} and Wang \cite{W11} studied the inverse problem of supersonic flow past two-dimensional curved wedges for a given shock position and incoming flow; Hu-Li-Zhang \cite{HLZ24} also considered the case of three-dimensional axisymmetric cones. Other inverse problems related to supersonic flow past obstacles have been investigated in \cite{CPZ25,Hayes59,Hida61,PZ23,V98}. However, as for the global flow field structure of the waverider, the relevant mathematical results have not yet been found even for the simplest Nonweiler wing.

In this paper, we establish the global flow field structure of the Nonweiler wing in a Chaplygin gas based on the result established in \cite{LY22} for supersonic flow over a delta wing. Specifically, starting from an ideal delta wing, namely a flat and infinite-span triangular plate, we analyze the evolution of flow field structures and shock positions during its downward folding along its ridge line. It is shown that there exists a critical configuration in which the attached shock of the conical wing with a $\Lambda$-shaped cross section becomes planar, thus verifying the existence of the flow field structure of the Nonweiler wing as illustrated in Figure \ref{fig1}.  Unlike the methods in \cite{HLZ24,LZ22,W11}, which would suggest an inverse‑problem approach, we solve this problem by adopting a direct‑problem perspective. Moreover, unlike those works that focus on boundary‑value problems for hyperbolic equations, we must consider a boundary‑value problem for a nonlinear mixed‑type equation (see Problem B below).

For the conical wing with $\Lambda$-shaped cross sections, K\"uchemann gave a speculation on the global flow field structures in \cite[p.304]{KD65}. For the Chaplygin gas, we will confirm the validity of two of these structures and analyze the non-existence of the other structure (see Section \ref{sec:possible} for details). This work will facilitate the study of waveriders, especially for conical configurations.

We remark here that the problem of shock reflection for two-dimensional unsteady self-similar flow arises in \cite{BCF09,BCF24,CF10,CF18,WZ23}, which also leads to a nonlinear mixed-type boundary value problem.

The Chaplygin gas, with state equation $p=a^2({1}/{\rho_*}-{1}/{\rho})$, has been applied to subsonic airfoil design \cite{T1939} and cosmology \cite{Popov10}, where $p$ and $\rho$ are unknowns representing the scalar pressure and density of mass, respectively; and $ a, \rho_*$ are two positive constants. In this gas, the multidimensional Riemann problem has been extensively studied in \cite{CQu12,CCW22,GSZ10,SL16,LSZ11,WG26}. For our study, the relevant properties of the Chaplygin gas established in \cite[Section 2]{Serre09} and \cite[Appendix A.1]{Serre11} are:

$(i)$ For a piecewise smooth three-dimensional steady compressible Euler flow, the flow maintains irrotational and isentropic across a shock;
			
$(ii)$ For the three-dimensional steady compressible Euler equations, all characteristics are linearly degenerate. This implies that shocks are characteristic and fluid particles cross the steady shock at the speed of sound.

%main
Now, let us describe our problem. Let $\mathcal{W}^\beta_\sigma$ denote an infinite-span conical wing with $\Lambda$-shaped cross sections, in which the angle between the leading edge and ridge line is $\frac{\pi}{2}-\sigma$ with $\sigma \in (0,\pi/2)$ being the sweep angle (see Figure \ref{fig2-x2ox3}); and the anhedral angle is $\beta$ with $\beta\in (0,\pi/2)$ (see Figure \ref{fig2-x1ox2}). Let $\bm{e}_1, \bm{e}_2, \bm{e}_3$ be the standard orthonormal basis of the Euclidean space $\mathbb{R}^3$, with a typical point $x=x_1\bm{e}_1+x_2\bm{e}_2+x_3\bm{e}_3$. In the $(x_1, x_2, x_3)$-coordinates, we symmetrically position $\mathcal{W}^\beta_\sigma$ with respect to the $x_1Ox_3$-plane, with the apex at the origin and the ridge line along the positive $x_3$-axis (see Figure \ref{fig2}), namely,
	\begin{align}\label{sweep wing}
		\mathcal{W}^\beta_\sigma=\{(x_1,x_2,x_3):x_1=-|x_2|\tan\beta, |x_2|<x_3\cot\sigma\cos\beta, x_3>0\}.
	\end{align}
Let $x_1 = s(x_2, x_3)$ be the equation for the shock attached to the leading edges of $\mathcal{W}^\beta_\sigma$ (cf. Figure \ref{fig2}). Define a region
	\begin{align*}
		\mathcal{R}^\beta_\sigma\doteq\{s(x_2,x_3)<x_1<-x_2\tan\beta,~x_2>0,~x_3>0\}.
	\end{align*}
The symmetry of $\mathcal{W}^\beta_\sigma$ allows us to consider the problem in the region $\mathcal{R}^\beta_\sigma$.

 %-----------------------------------fig-----------
\begin{figure}[htb]
\centering
\includegraphics[scale=1.1]{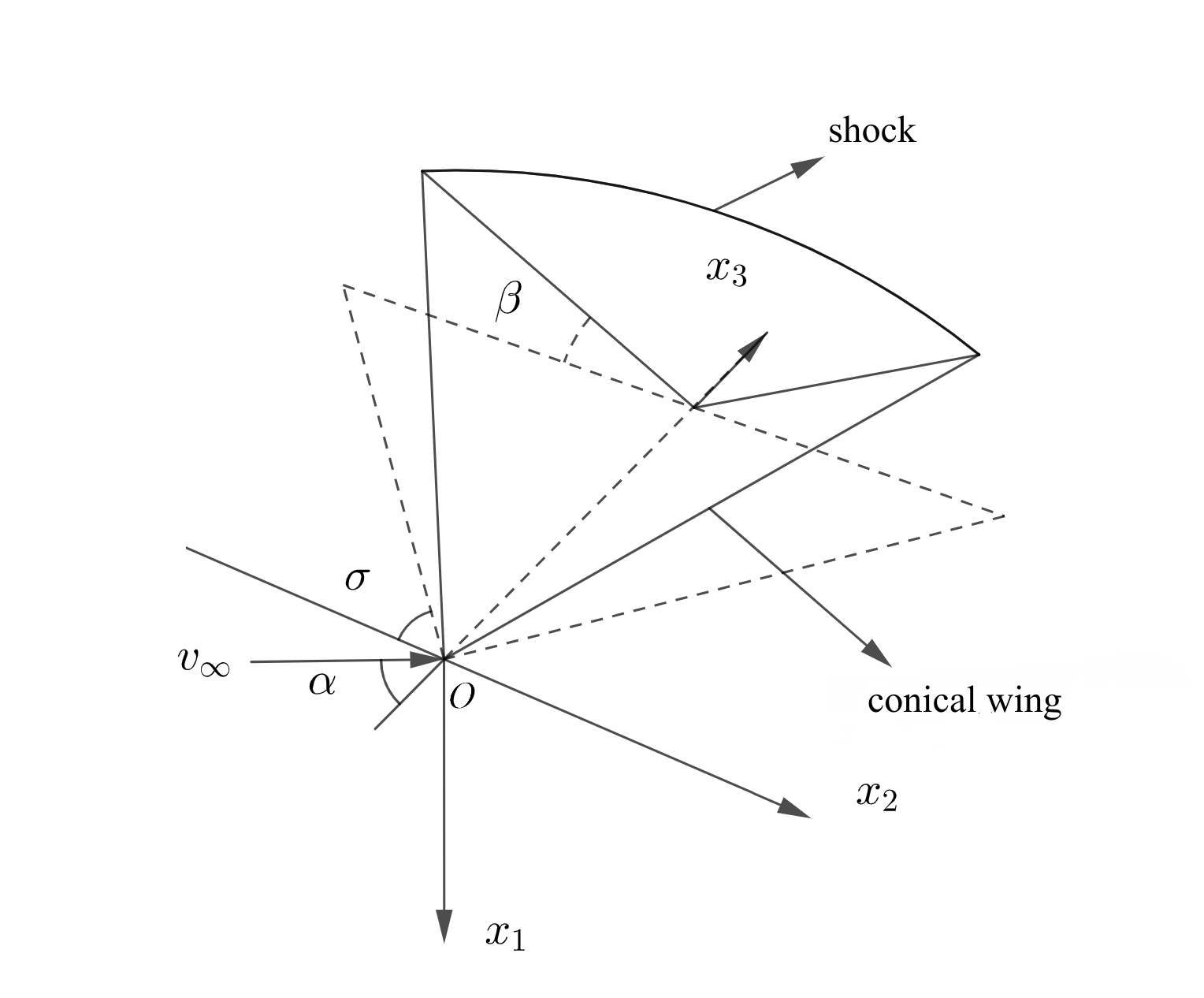}
\caption{A conical wing with $\Lambda$-shaped cross section.}\label{fig2}
\end{figure}
 %-----------------------------------fig-----------
 \begin{figure}[htb]
 	\centering
 	\includegraphics[scale=0.7]{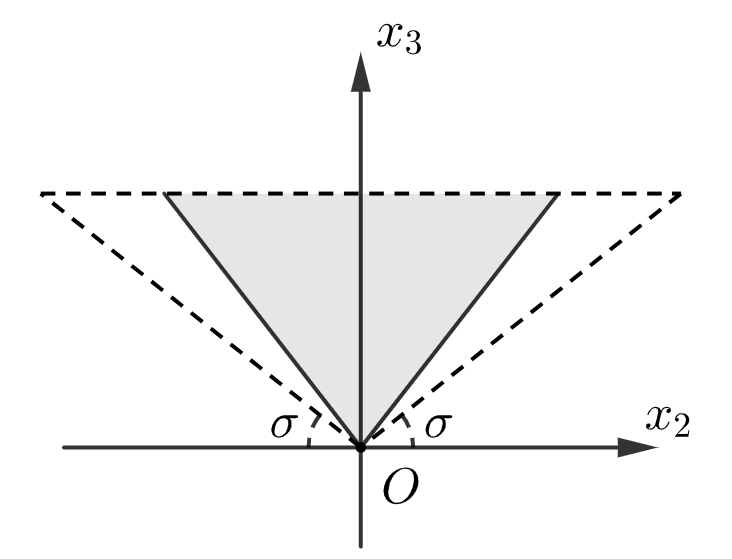}
 	\caption{View of a conical wing from the $x_1$--direction, where $\sigma$ is the sweep angle.}\label{fig2-x2ox3}
 \end{figure}
  \begin{figure}[htb]
 	\centering
 	\includegraphics[scale=0.5]{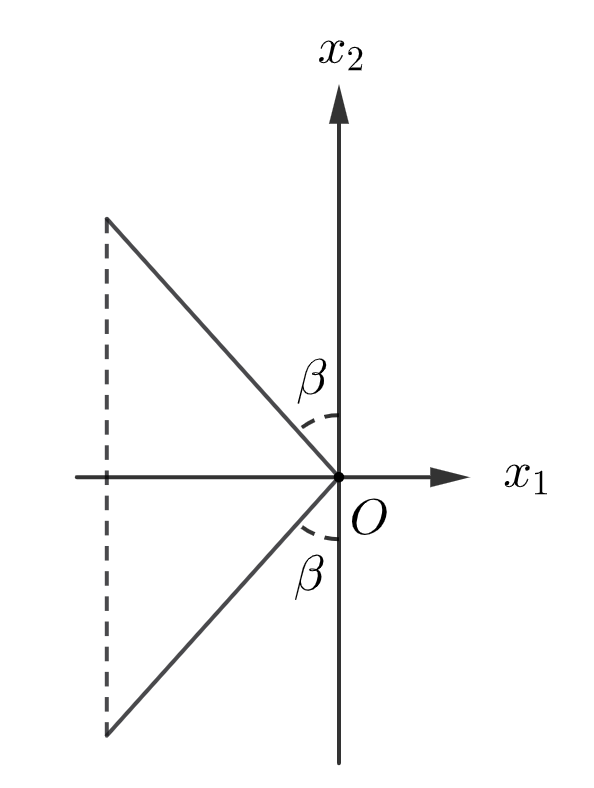}
 	\caption{View of a conical wing from the $x_3$--direction, where $\beta$ is the anhedral angle.}\label{fig2-x1ox2}
 \end{figure}

\begin{figure}[htb]
	\centering
	\includegraphics[scale=0.6]{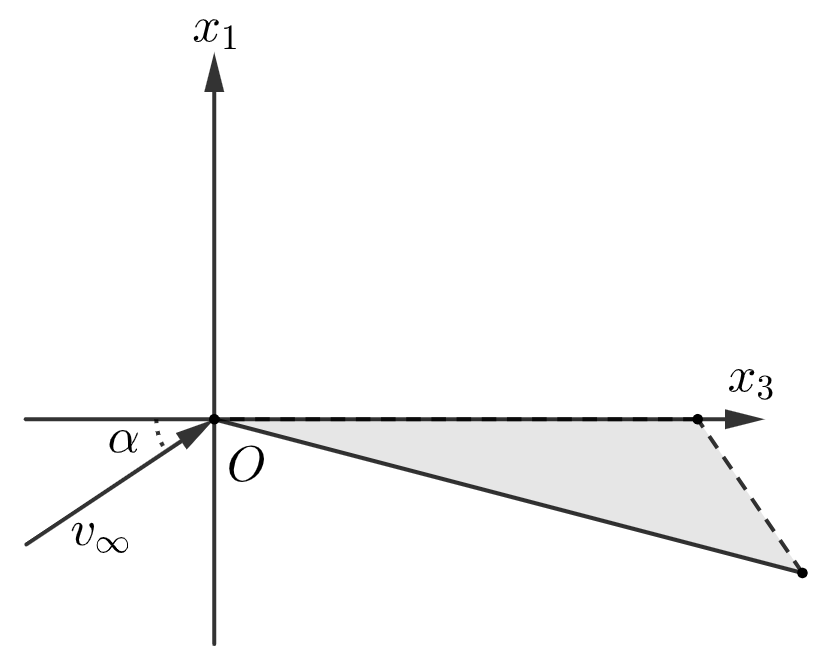}
	\caption{View of a conical wing from the $x_2$--direction, where $\alpha$ is the attack angle.}\label{fig2-x1ox3}
\end{figure}
 
The incoming flow with uniform state
   \begin{equation}
        \label{incoming flow}
       U_\infty=(\rho_{\infty},q_{\infty})
    \end{equation} 
is assumed to be supersonic, passing the wing $\mathcal{W}^\beta_\sigma$ with an attack angle $\alpha$, where $\alpha \in (0, \pi/2)$ (see Figure \ref{fig2-x1ox3}). This means that the velocity of the incoming flow is $\bm{v}_\infty \doteq (v_{1\infty}, 0, v_{3\infty})=(q_\infty \sin \alpha, 0, q_\infty \cos \alpha)$. By property $(i)$ of the Chaplygin gas above, the flow behind the shock is exactly potential flow. The three-dimensional steady compressible Euler equations for potential flow are given by
 \begin{align}\label{potential}
	\mathrm{div}_{\bm{x}}(\rho\nabla_{\bm{x}}\Phi)=0,
    \end{align}
and
	\begin{align}\label{bernoulli}
	\frac{1}{2}|\nabla_{\bm{x}}\Phi|^2+h(\rho)=B_\infty,
	\end{align}
where $\Phi$ is the velocity potential satisfying $\nabla _{\bm{x}}\Phi=\bm{v}$ with $\bm{x}\doteq (x_1, x_2, x_3)$; $h(\rho)\doteq-a^2/(2\rho^2)$ is the specific enthalpy; $B_\infty=(q^2_\infty-c^2_\infty)/2$ is the Bernoulli constant with $c_\infty$ being the local speed of sound corresponding to the incoming flow. Equation \eqref{potential} describes the law of conservation of mass, while equation \eqref{bernoulli} is Bernoulli's law. Without loss of generality, we set $B_{\infty}=\frac{1}{2}$ throughout the subsequent analysis. 

Notice that since any shock is a characteristic for the Chaplygin gas, the Rankine–Hugoniot condition $[\kern-0.1em[\rho \nabla_{\bm{x}}\Phi]\kern-0.1em]\cdot\bm{n}_{s}=0$ holds automatically on the shock, where the bracket $[\kern-0.1em[\cdot]\kern-0.1em]$ denotes the jump of quantities from one side of the discontinuity  to the other, and $\bm{n}_{s}$ is a unit normal vector to the shock. Then, the flow satisfies the following Dirichlet boundary condition 
	\begin{align}\label{continuec}
		\Phi=\Phi_\infty\doteq v_{1\infty} x_1 + v_{3\infty}x_3
	\end{align}
on the shock, and the following slip boundary conditions
	\begin{align}
		&\nabla_x\Phi\cdot \bm{n}_w=0 \quad \text{on}\quad \{x_1=-x_2\tan\beta,~ x_2>0\} ,\label{bcwing}\\
		&\nabla_x\Phi\cdot \bm{n}_{sy}=0 \quad \text{on}\quad \{x_2=0\},\label{bcsym}
	\end{align}
	where $\bm{n}_w = (\cos\beta, \sin\beta, 0)$ is the exterior unit normal to $\{x_1 = -x_2\tan\beta,~x_2>0\}$, and $\bm{n}_{sy} = (0, -1, 0)$ is the exterior unit normal to $\{x_2 = 0\}$.  
	
Therefore, our problem can be formulated mathematically as
 %%%%%%%%%%%%%%%%%%%%%%%%%%%%%%%%%%%%
\vspace{\baselineskip} % 段前间距    
\begin{center}    
\fbox{\begin{varwidth}{0.9\linewidth} % 设置宽度为行宽的80%        
\centering \textbf{Problem A}: For the wing $\mathcal{W}^\beta_\sigma$ and incoming flow $U_{\infty}$ given by \eqref{sweep wing} and \eqref{incoming flow} respectively, find a solution to \eqref{potential}--\eqref{bernoulli} in the region $R^\beta_\sigma$ with the Dirichlet condition \eqref{continuec} and slip conditions \eqref{bcwing}--\eqref{bcsym}. % 内容居中    
\end{varwidth}}    
\end{center}    
\vspace{\baselineskip} % 段后间距
%%%%%%%%%%%%%%%%%%%%%%%%%%%%%%%%%%%%% 
	
The following theorem is the main result of this paper.
	\begin{theorem}\label{thm: Main}
	    Assume that the wing $\mathcal{W}^\beta_\sigma$ and incoming flow $U_{\infty}$ given by $(\ref{sweep wing})$ and $(\ref{incoming flow})$, respectively. Then, we can find a critical attack angle $\alpha_0 = \alpha_0(\rho_\infty, q_\infty) \in (0, \pi/2)$ such that for any fixed $\alpha \in (0, \alpha_0)$, there exists a critical sweep angle $\sigma_0 = \sigma_0(\rho_\infty, q_\infty, \alpha) \in (0, \pi/2)$ such that for $\sigma \in(0, \sigma_0]$, there is a critical anhedral angle $\beta_0 = \beta_0(\rho_\infty, q_\infty, \alpha,\sigma) \in [0, \pi/2)$, and Problem A admits a piecewise smooth solution for all $\beta\in[0,\beta_0]$.
	\end{theorem}

   \begin{remark}\rm
       For a given incoming flow $U_\infty$, mass concentration will occur when $\alpha \geq \alpha_0$, as discussed in Section \ref{sub2.1} below. Besides, for a given $U_\infty$ and $\alpha \in (0, \alpha_0)$, when $\sigma >\sigma_0$, the shock will detach from the leading edge of $\mathcal{W}^\beta_\sigma$ and intersect it only at its apex (see \cite{L25} for more detail), which conflicts with the waverider design requirement that the shock should remain attached. As for the case $\beta>\beta_0$, the corresponding analysis is more complex and will be investigated in future work. Here, these critical angles will be defined by \eqref{alpha0}--\eqref{sigma0} and \eqref{debeta0} below. \qed
    \end{remark}

   \begin{remark}\rm
         For the analysis of Problem A without the requirement that the incoming flow velocity $v_{2\infty}=0$, see Section \ref{sec:Asywing} for further details. Moreover, for a thin Nonweiler wing as in Figure $\ref{fig1}$, the analysis is completely analogous to the case of $\mathcal{W}^\beta_\sigma$, and the corresponding result coincides with that in Theorem $\ref{thm: Main}$. Thus, we omit the relevant details. \qed
    \end{remark}
    
    \begin{remark}\rm
        For any fixed $\alpha \in (0, \alpha_0)$ and $\sigma \in(0, \sigma_0]$, there always exists an angle $\beta_c=\beta_c(\rho_\infty, q_\infty, \alpha,\sigma)\in(0,\beta_0)$ such that the attached shock becomes planar; that is, the structure of the Nonweiler wing as illustrated in Figure $\ref{fig1}$ is reasonable, see Lemma \ref{lemma22}. \qed
    \end{remark}

    \begin{remark}\rm 
       The anhedral angle  $\beta$  is a key parameter in the design of supersonic aircraft. To the best of our knowledge, this parameter has not been considered in previous mathematical  studies. \qed
    \end{remark}

The remainder of this paper is organized as follows. In Section \ref{sec2pre}, we first derive the uniform downstream states outside the Mach cone by applying the shock polar for the Chaplygin gas. We subsequently determine the structures of the shocks in the pseudo-self-similar $(\xi_1,\xi_2)$-coordinates, which facilitates the reformulation of Problem A into a boundary value problem for a nonlinear mixed-type equation, denoted as Problem B. Section \ref{sec:3} focuses on the flow inside the Mach cone through the analysis of Problem B. Specifically, by introducing two key parameters, we consider the problem \eqref{eq: for Fmu}--\eqref{eq:nianxing BC} below instead of Problem B. By establishing a crucial Lipschitz estimate, we prove the existence of the viscosity solution to the problem \eqref{eq: for Fmu}--\eqref{eq:nianxing BC} via the continuity method, and then show that the viscosity solution converges to the solution of Problem B. Section \ref{sec:Further} gives a brief discussion on K\"uchemann's speculation for the global conical flow field structures, and also analyzes the problem of supersonic flow over an asymmetric conical wing with $\Lambda$-shaped cross sections. Appendix \ref{appendix a} recalls the shock polar for the Chaplygin gas.

\section{Analysis of the flow outside Mach cone and the shock structures}\label{sec2pre}

When the conical wing $\mathcal{W}^\beta_\sigma$ is simplified to the case $\beta = 0$, Problem A is reduced to the problem of supersonic flow past an ideal delta wing, and the existence of piecewise smooth solutions has been established in \cite[Theorem 1.2]{LY22}. Thus, in what follows, we mainly focus on $\beta>0$.
	
\subsection{Uniform downstream flow outside the Mach cone}\label{sub2.1}
From the Bernoulli's law \eqref{bernoulli}, it follows that the density $\rho$ can be determined explicitly in terms of the velocity potential $\Phi$ as
	\begin{align}\label{rhoPhi}
	\rho=\frac{a}{\sqrt{|\nabla_{\bm{x}}\Phi|^2-1}}.
	\end{align}
Inserting \eqref{rhoPhi} into \eqref{potential}, we derive a quasilinear equation for $\Phi$: 
	\begin{align}\label{second order}
		(c^2-\Phi^2_{x_1})\Phi_{x_1x_1}+&(c^2-\Phi^2_{x_2})\Phi_{x_2x_2}+(c^2-\Phi^2_{x_3})\Phi_{x_3x_3}\nonumber\\&-2\Phi_{x_1}\Phi_{x_2}\Phi_{x_1x_2}-2\Phi_{x_1}\Phi_{x_3}\Phi_{x_1x_3}-2\Phi_{x_2}\Phi_{x_3}\Phi_{x_2x_3}=0.
	\end{align}
The type of this equation is characterized by its characteristic form
	\begin{align*}
		Q(\bm{\zeta})=c^2-|\nabla_{\bm{x}}\Phi\cdot\bm{\zeta}|^2\quad \text{for any}\quad \bm{\zeta} \in \mathbb{R}^3, ~|\bm{\zeta}|=1,
	\end{align*}
which indicates that \eqref{second order} is hyperbolic in regions of supersonic flow and elliptic in regions of subsonic flow. Recall that the normal component of the flow velocity across the shock is sonic (i.e., the property $(ii)$ of the Chaplygin gas listed in Section \ref{sec1}). This indicates that the flow behind the shock is supersonic, of which equation \eqref{second order} is then hyperbolic, owing to the supersonic incoming flow. Therefore, there exists a Mach cone emanating from the apex of $\mathcal{W}^\beta_\sigma$, such that the solution to Problem A remains undisturbed outside the Mach cone. Moreover, outside the Mach cone, the flow behind the shock is uniform and the corresponding attached shock is flat. In the following, we denote the flat shock by $S^\beta_{ob}$. 

The main goal of this subsection is to calculate the explicit form of the solution to Problem A outside the Mach cone. To decompose the incoming flow velocity along the direction perpendicular to the leading edge of $\mathcal{W}^\beta_{\sigma}$ and subsequently apply the shock polar to calculate the uniform downstream flow state, we introduce an orthonormal basis \(\{\bm{e}_i, \bm{e}_j, \bm{e}_k\}\) adapted to the geometry of the wing \(\mathcal{W}^\beta_{\sigma}\). Specifically, \(\bm{e}_i\) is chosen as the unit normal to \(\mathcal{W}^\beta_{\sigma}\), \(\bm{e}_j\) as the unit vector along its leading edge, and \(\bm{e}_k = \bm{e}_i \times \bm{e}_j\) as the unit tangent perpendicular to the leading edge. This basis given explicitly by
\[
\begin{aligned}
\bm{e}_i &\doteq (\cos \beta,\; \sin \beta,\; 0), \\
\bm{e}_j &\doteq (-\cos\sigma \sin\beta,\; \cos\sigma \cos\beta,\; \sin\sigma), \\
\bm{e}_k &\doteq \bm{e}_i \times \bm{e}_j = (\sin\sigma \sin\beta,\; -\sin\sigma \cos\beta,\; \cos\sigma),
\end{aligned}
\]
enables a natural decomposition of the incoming flow velocity $\bm{v}_\infty$ that 
\begin{align}\label{vinfcomponent}
\bm{v}_\infty&=v_{1\infty}\bm{e}_1+v_{2\infty}\bm{e}_2+v_{3\infty}\bm{e}_3\nonumber\\
		&=v_{1\infty}\cos\beta \bm{e}_i+(v_{3\infty}\sin\sigma-v_{1\infty}\cos\sigma\sin\beta)\bm{e}_j+(v_{1\infty}\sin\sigma\sin\beta+v_{3\infty}\cos\sigma)\bm{e}_k,
	\end{align}
where we introduce
	\begin{align*}
		\tilde{\bm{v}}_\infty\doteq v_{1\infty}\cos\beta \bm{e}_i+(v_{1\infty}\sin\sigma\sin\beta+v_{3\infty}\cos\sigma)\bm{e}_k
	\end{align*}
    as the velocity component of the incoming flow perpendicular to the leading edge of the wing. The magnitude of $\tilde{\bm{v}}_\infty$ is
    	\begin{align}\label{eqqinftytilde}
    		\tilde{q}_\infty \doteq\sqrt{v^2_{1\infty}\cos^2\beta+(v_{1\infty}\sin\sigma\sin\beta+v_{3\infty}\cos\sigma)^2},
    	\end{align}
    and the angle between $\tilde{\bm{v}}_\infty$ and $\bm{e}_k$ is
    	\begin{align*}
    	\theta_n \doteq \arctan\left(\frac{\cos\beta}{\sin\sigma\sin\beta+\cot\alpha\cos\sigma}\right).
    	\end{align*}
Then, the uniform downstream flow velocity \(\bm{v}^\beta_\sigma = (v^\beta_{1\sigma}, v^\beta_{2\sigma}, v^\beta_{3\sigma})\) can be determined via the shock polar relation for the Chaplygin gas. More precisely, this velocity component along \(\bm{e}_k\), denoted by \(q^\beta_{j\sigma}\), is obtained by setting \(u_0 = \tilde{q}_\infty\), \(c_0 = c_\infty\) and \(\theta = \theta_n\) in \eqref{u1v1polar}. Meanwhile, the component along \(\bm{e}_j\) remains unchanged upon crossing the flat shock \(S^\beta_{ob}\) due to the constancy of the tangential velocity. Hence, \(\bm{v}^\beta_\sigma\) admits the decomposition
\[
\bm{v}^\beta_\sigma = \big(v_{3\infty}\sin\sigma - v_{1\infty}\cos\sigma \sin\beta\big)\,\bm{e}_j + q^\beta_{j\sigma}\,\bm{e}_k,
\]
which, upon expansion in the Cartesian coordinates, yields the explicit expressions:
\begin{align}
v^\beta_{1\sigma} &= v_{1\infty}\cos^2\sigma\sin^2\beta - v_{3\infty}\sin\sigma\cos\sigma\sin\beta + q^\beta_{j\sigma}\sin\sigma\sin\beta, \label{v1sigma}\\
v^\beta_{2\sigma} &= -v_{1\infty}\cos^2\sigma\sin\beta\cos\beta + v_{3\infty}\sin\sigma\cos\sigma\cos\beta - q^\beta_{j\sigma}\sin\sigma\cos\beta, \label{v2sigma}\\
v^\beta_{3\sigma} &= -v_{1\infty}\cos\sigma\sin\sigma\sin\beta + v_{3\infty}\sin^2\sigma + q^\beta_{j\sigma}\cos\sigma. \label{v3sigma}
\end{align}
Obviously, with \(\bm{v}^\beta_\sigma\) explicitly known, the corresponding velocity potential outside the Mach cone is given by
\[
\Phi^\beta_\sigma = v^\beta_{1\sigma}x_1 + v^\beta_{2\sigma}x_2 + v^\beta_{3\sigma}x_3,
\]
which satisfies equation \eqref{second order} with the boundary conditions \eqref{continuec}--\eqref{bcwing}. The associated sound speed follows from the Bernoulli law \eqref{bernoulli} as $c^\beta_\sigma = \sqrt{|\nabla_{\bm{x}}\Phi^\beta_\sigma|^2 - 1}$.

At the end of this subsection, we focus on the roles of the angles $\alpha$, $\sigma$ and $\beta$ since excessive variations in these angles will induce a phenomenon called concentration. To avoid this occurrence, it follows from \eqref{restrictpolar} and the discussion in \cite[Appendix A]{LY22} that
    \begin{align}\label{restict}
        c_\infty<\tilde{q}_\infty<\frac{c_\infty}{\sin\theta_n}.
    \end{align}
Noticing that $\tilde{q}_\infty\sin\theta_n=v_{1\infty}\cos\beta$ and $v_{1\infty}=q_\infty\sin\alpha$, and using the right-hand side of \eqref{restict}, one gets
    \begin{align}\label{rightineq}
    q_\infty\sin\alpha\cos\beta<c_\infty.    
    \end{align}
It is obvious that the above relation holds for all $\beta\in[0,\pi/2)$ if it holds at $\beta=0$; namely, for fixed $\alpha$ and $\sigma$, the phenomenon of concentration occurs only at $\beta = 0$.

Besides, from \eqref{vinfcomponent} and \eqref{eqqinftytilde}, the component of velocity $\bm{v}_\infty$ along the $\bm{e}_j$ direction decreases with respect to $\beta$; that is, $\tilde{q}_\infty$  increases with respect to $\beta$. Thus, the left-hand side inequality of \eqref{restict} is valid for $\beta\in[0,\pi/2)$ once it holds at $\beta=0$. 

In conclusion, we only need to consider the case $\beta=0$. For this case, from the discussion in \cite[Remark 2.1]{LY22}, and using \eqref{rightineq}, one knows that for any $\alpha \in(0, \alpha_0)$, with 
\begin{align}\label{alpha0}
        \alpha_0\doteq\arcsin\left(\frac{c_\infty}{q_\infty}\right),
    \end{align}
the phenomenon of concentration would never occur. Moreover, the left-hand side of \eqref{restict} implies the sweep angle $\sigma<\sigma_0$ with
    \begin{align}\label{sigma0}
        \sigma_0\doteq\arcsin\left(\frac{\sqrt{q^2_\infty-c^2_\infty}}{v_{3\infty}}\right).
    \end{align}

    \begin{remark}\rm 
    The condition $\sigma<\sigma_0$ is sufficient to ensure that the shock remains attached to the leading edge of the conical wing $\mathcal{W}^\beta_\sigma$. \qed
    \end{remark}

\subsection{Structures of shocks in conical coordinates}\label{sec2.2}	 

It is shown in \cite[Section 2.2]{LY22} that for any fixed $\alpha \in (0, \alpha_0)$, the attached shock appears for $\sigma \in (0, \sigma_0]$, where $\alpha_0$ and $\sigma_0$ are given by \eqref{alpha0} and \eqref{sigma0}, respectively. From now on, we fix $\alpha \in (0, \alpha_0)$ and $\sigma \in (0, \sigma_0]$. We will further show that there exists a critical angle $\beta_0$ such that the attached shock always exists for $\beta\in(0,\beta_0]$. In addition, the structures of shocks admit an explicit representation when expressed in the conical coordinates introduced below.
	
Since the problem \eqref{second order} and \eqref{continuec}--\eqref{bcsym} is invariant with respect to the following scaling
	\begin{align}\label{scaling}
		\bm{x}\rightarrow\tau \bm{x},~ (\rho,\Phi)\rightarrow(\rho, \frac{\Phi}{\tau} ) \quad \text{for} \quad \tau\neq 0,
	\end{align}
we can introduce the conical coordinates $\bm{\xi}=(\xi_1,\xi_2) \doteq (x_1/x_3, x_2/x_3)$ to find a self-similar solution of the form:
	\begin{align}\label{conical}
		\rho(\bm{x})=\rho(\xi_1,\xi_2), \quad\Phi(\bm{x})=x_3\phi(\xi_1,\xi_2).
	\end{align}
		
To proceed, we give the following notations. Let $\mathcal{C}_\infty$ and $\mathcal{C}^\beta_\sigma$ denote the Mach cones of the wing apex, determined respectively by the incoming flow and uniform downstream flow. By an abuse of notation that causes no ambiguity, in the conical coordinates, we still use $\mathcal{C}_\infty$ and $\mathcal{C}^\beta_\sigma$ to denote the corresponding curves of the Mach cones, and $S^\beta_{ob}$ to denote the corresponding oblique shock respectively. 

We first present the equations for the oblique shock \(S^\beta_{ob}\) and the Mach cones \(\mathcal{C}_\infty\), \(\mathcal{C}^\beta_\sigma\) in the conical coordinates, given respectively by
\begin{align}
S^\beta_{ob}:&\quad v_{1\infty}\xi_1+v_{3\infty}=v^\beta_{1\sigma}\xi_1+v^\beta_{2\sigma}\xi_2+v^\beta_{3\sigma},\label{sob}\\
\mathcal{C}_\infty:&\quad (v_{1\infty}\xi_1+v_{3\infty})^2=1+|\bm{\xi}|^2,\label{cinf}\\
\mathcal{C}^\beta_\sigma:&\quad (v^\beta_{1\sigma}\xi_1+v^\beta_{2\sigma}\xi_2+v^\beta_{3\sigma})^2=1+|\bm{\xi}|^2.\label{csigma}
\end{align}
These equations are derived as follows. The shock equation \eqref{sob} is obtained from the continuity of the velocity potential \(\Phi\) across the shock. Indeed, applying the scaling relation \eqref{scaling}, together with the explicit expressions for \(\Phi_\infty\) and \(\Phi^\beta_\sigma\), one obtains
\begin{align}
		&\phi_\infty=v_{1\infty}\xi_1+v_{3\infty},\label{phiinf}\\
&\phi^\beta_\sigma=v^\beta_{1\sigma}\xi_1+v^\beta_{2\sigma}\xi_2+v^\beta_{3\sigma}.\label{phisigma}
	\end{align}
Then the condition \(\phi_\infty = \phi^\beta_\sigma\) on \(S^\beta_{ob}\) leads directly to \eqref{sob}.

Turning to the Mach cones, we note that in the conical coordinates, their equations follow from the general form introduced in \cite[(B.4)]{LY22} reformulated into the form
\begin{equation}\label{machconecc}
	|\mathrm{D}\phi|^{2}+|\phi-\mathrm{D}\phi\cdot\bm{\xi}|^{2}-\frac{\phi^{2}}{1+|\bm{\xi}|^{2}}=c^{2},
\end{equation}
which can be further simplified to
\begin{align}\label{phi2}
		\phi^2=1+|\bm{\xi}|^2,
\end{align}
by virtue of the relation $c^2=|\mathrm{D}\phi|^2+|\phi-\mathrm{D}\phi\cdot\bm{\xi}|^2-1$ derived from \eqref{rhoPhi}, \eqref{conical} and the constitutive relation \(\rho c = a\). Subsequently, substituting the conical potentials \(\phi_\infty\) and \(\phi^\beta_\sigma\) from \eqref{phiinf}–\eqref{phisigma} into \eqref{phi2} yields \eqref{cinf} and \eqref{csigma}. It is worth noting that these Mach cones are not necessarily circular in the conical coordinates.
	
Next, we consider the global structures of shocks. From the property that any shock is a characteristic, the oblique shock $S^\beta_{ob}$ is required to be tangent to the curve $\mathcal{C}_\infty$ at a point denoted by $P^\beta_1$. This tangency point, together with other key intersection points defined below, determines the geometric configurations (cf. Figure \ref{fgbetageq0} or Figure \ref{figbeta0} for $\beta=0$). Let \begin{align}\label{gammawing}
    \Gamma_{wing}:~~ \xi_1\cos\beta+\xi_2\sin\beta=0~~(\xi_2\geq0),
\end{align}
denote the conical wing $ \mathcal{W}^\beta_\sigma$ with  $x_2\geq0$ in the $(\xi_1,\xi_2)$-coordinates. The intersections of $\mathcal{C}_\infty$ with $\Gamma_{wing}$ and the negative $\xi_1$-axis are denoted by $P^\beta_0$ and $P_2$ respectively; those of $\mathcal{C}^\beta_\sigma$ and $S^\beta_{ob}$ with $\Gamma_{wing}$ are denoted by $P^\beta_4$ and $P^\beta_5$ respectively; and the intersection of the extension line of $S^\beta_{ob}$ with the $\xi_1$--axis is denoted by $P^\beta_6$.

Based on the intersection points above, we denote \(\Gamma^\infty_{\mathrm{cone}}\) by the arc \(\widehat{P^\beta_1P_2}\) and \(\Gamma^\beta_{\mathrm{cone}}\) by the arc \(\widehat{P^\beta_1P^\beta_4}\) (cf. Figure \ref{figbeta0} for $\beta=0$). The line segment \(\Gamma_{\mathrm{sym}}\) is defined as \(OP_2\), and the shock curve \(\Gamma_{\mathrm{shock}}\) is given by the union \(S^\beta_{ob}\cup\Gamma^\infty_{\mathrm{cone}}\). Then, we can introduce the domains \(U\) and \(\Omega\) for subsequent analysis, with \(U\) representing the region \(OP_2P^\beta_1P^\beta_5\) and \(\Omega\) representing the region \(OP_2P^\beta_1P^\beta_4\).

 %-----------------------------------fig-----------
\begin{figure}[htb]
\centering
\includegraphics[scale=1.46]{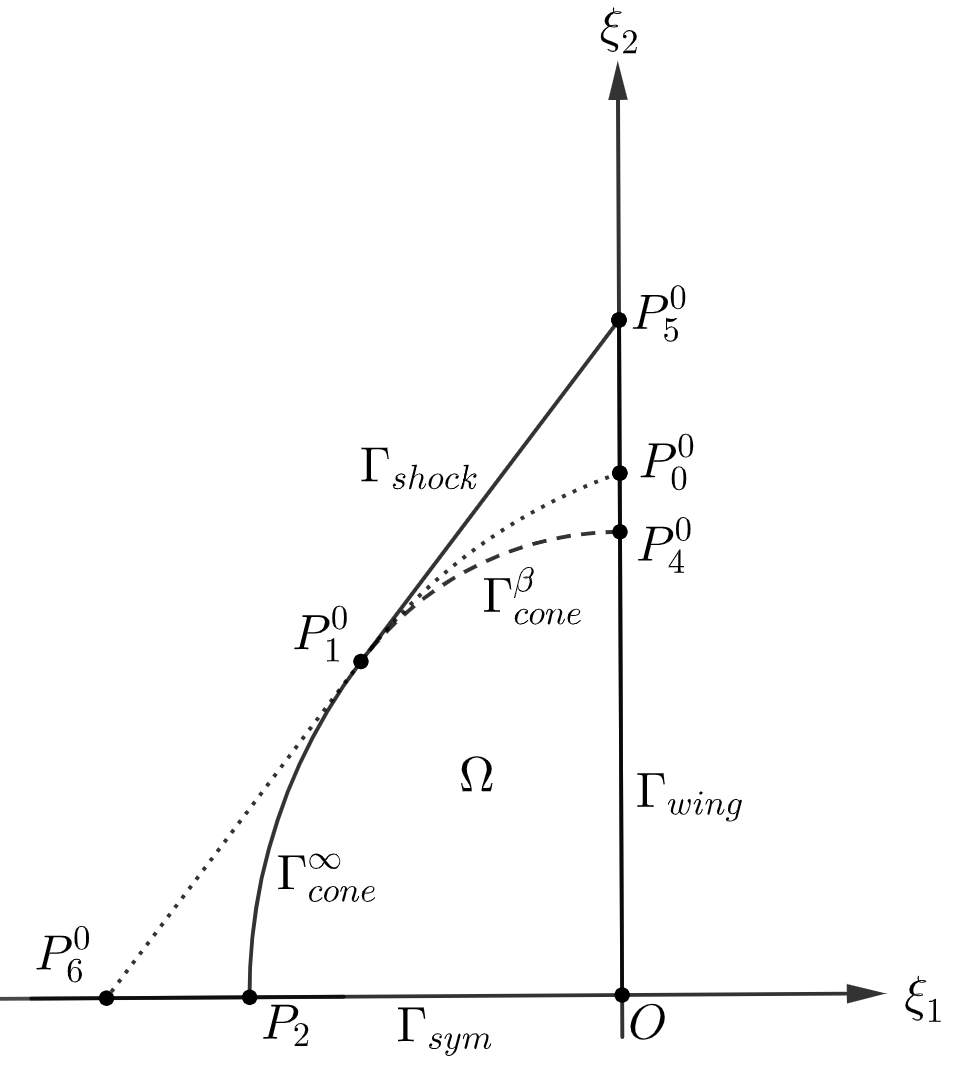}
\caption{The case for $\beta=0$.}\label{figbeta0}
\end{figure}
 %-----------------------------------fig----------- 
We recall that for the case $\beta=0$, the pattern of shock waves was discussed in \cite[Section 2.2]{LY22}, and Figure \ref{figbeta0} is redrawn from \cite[Figure 4(a)]{LY22}. It should also be pointed out that the points $P^0_4$ and $P^0_5$ meet at $ P^0_0$ when $\sigma = \sigma_0$, where $P^0_i$ denotes $P^\beta_i$ with $\beta=0$ for $i=0,1,4,5,6$. Thus, we only need to consider the case $\beta>0$ in the following.

\begin{lemma}\label{lemma22}
Let $\mathcal{C}_\infty, \ \mathcal{C}^\beta_\sigma$ be defined as \eqref{cinf}--\eqref{csigma} with $\beta>0$. For any fixed $\alpha \in(0, \alpha_0)$ and $\sigma \in (0, \sigma_0]$, there exists a critical angle 
    \begin{align}\label{eqdefinebetac}
        \beta_c\doteq\arcsin(-\xi_{1P_2}/\cot\sigma),
%=\arctan(\frac{2v_{1\infty}v_{3\infty}+\sqrt{4v^2_{1\infty}v^2_{3\infty}-4(v^2_{1\infty}-1)(v^2_{3\infty}-1)}{2(v^2_{1\infty}-1)\cot\sigma}),
    \end{align}
so that the point $P^\beta_1$ coincides with $P_2$, and the curved shock wave vanishes; namely, the attached shock becomes planar, where $\xi_{1P_2}<0$ denotes the $\xi_1$-coordinate of $P_2$ determined by \eqref{cinf} with $\xi_2=0$, and $\alpha_0$, $\sigma_0$ are given by \eqref{alpha0}, \eqref{sigma0}, respectively.
	\end{lemma}
\begin{proof}
Let us first demonstrate that $\mathcal{C}_{\infty}$ and $\mathcal{C}^\beta_{\sigma}$ are only tangent at the point $P^\beta_1$ when $\beta\in(0,\beta_c)$, where $\beta_c$ is to be determined such that $P^\beta_1$ coincides with $P_2$.
Note that the intersection points of \(\mathcal{C}_\infty\) and \(\mathcal{C}^\beta_\sigma\) are governed by the equation 
\begin{equation}\label{eqintercc}
|v_{1\infty}\xi_1+v_{3\infty}|=|v^\beta_{1\sigma}\xi_1+v^\beta_{2\sigma}\xi_2+v^\beta_{3\sigma}|,
	\end{equation}
which follows directly from \eqref{cinf} and \eqref{csigma}. 
%Let us first show that $\mathcal{C}_{\infty}$ and $\mathcal{C}^\beta_{\sigma}$ are tangent at the point $P^\beta_1$ when $\beta\in(0,\beta_c)$, where $\beta_c$ is to be determined such that $P^\beta_1$ coincides with $P_2$. From \eqref{cinf} and \eqref{csigma}, we see that the intersection points of $\mathcal{C}_{\infty}$ and $\mathcal{C}^\beta_{\sigma}$ satisfy
	%\begin{equation}\label{eqintercc}
%|v_{1\infty}\xi_1+v_{3\infty}|=|v^\beta_{1\sigma}\xi_1+v^\beta_{2\sigma}\xi_2+v^\beta_{3\sigma}|.
	%\end{equation}
    Let $P^\beta$ be the intersection point of the extension line of $S^\beta_{ob}$ and the $\xi_2$-axis. Obviously, $P^\beta_6$ lies on the right of $P^0_6$ when $\beta\in(0,\beta_c)$. Moreover, $\phi_\infty >0$ at $P^0_6$ follows from \cite[Lemma 2.2]{LY22}; namely, $v_{1\infty}\xi_{1P^0_6}+v_{3\infty}>0$, where $\xi_{1P^0_6}$ is the $\xi_1$-coordinate of $P^0_6$. Thus, in the triangle $\Delta OP^\beta_6P^\beta$, we have
	\begin{equation}\label{eqphiinfty}
		\phi_{\infty}(\bm{\xi})\geq v_{1\infty}\xi_{P^\beta_6}+v_{3\infty}\geq v_{1\infty}\xi_{1P^0_6}+v_{3\infty}>0,
	\end{equation}
which implies $\phi_{\infty}>0$ on $\mathcal{C}_{\infty}\cap \{\xi_2\geq0,\xi_1\leq0\}$. On the other hand, noting that $v^\beta_{1\sigma}\geq 0$ and $v^\beta_{2\sigma}\leq 0$, we deduce that the minimum value of $\phi^\beta_\sigma$ is attained on $P^\beta_6P^\beta$. Combined with this fact and \eqref{sob}, \eqref{eqphiinfty}, we obtain $\phi^\beta_\sigma>0$ on $\mathcal{C}^\beta_{\sigma} \cap \{\xi_2\geq0,\xi_1\leq 0\}$. In conclusion, the relation \eqref{eqintercc} can be reduced to 
	\begin{equation*}
v_{1\infty}\xi_1+v_{3\infty}=v^\beta_{1\sigma}\xi_1+v^\beta_{2\sigma}\xi_2+v^\beta_{3\sigma},
	\end{equation*}
which is the equation for the flat shock $S^\beta_{ob}$. Note that $P^\beta_1$ is also the intersection point of $\mathcal{C}_{\infty}$ and $S^\beta_{ob}$. Therefore, $P^\beta_1$ is the only intersection point of $\mathcal{C}_{\infty}$ and $\mathcal{C}^\beta_{\sigma}$.

We then determine the position of $P^\beta_5$ and $P^\beta_0$. Using \eqref{sob} and \eqref{gammawing}, one has 
\begin{align*}
    \xi_{2P^\beta_5}=\frac{v^\beta_{3\sigma}-v_{3\infty}}{v^\beta_{1\sigma}\tan\beta-v^\beta_{2\sigma}-v_{1\infty}\tan\beta}=\cot\sigma\cos\beta,
\end{align*}
where $\xi_{2P^\beta_5}$ denotes the $\xi_2$-coordinate of $P^\beta_5$. That is,  $|OP^\beta_5|=\cot\sigma$ on $\Gamma_{wing}$. Using the relation $\phi_\infty>0$ on $\mathcal{C}_{\infty}\cap \{\xi_2\geq0,\xi_1\leq 0\}$, and noting $v_{1\infty}>0$, we deduce from \eqref{cinf} and the definition of $P^\beta_0$ that $|OP^\beta_0|\leq|OP^0_0|$. This implies
\begin{align}\label{p0p5}
    |OP^\beta_0|<|OP^0_0|\leq|OP^0_5|=|OP^\beta_5|,
\end{align}
where the equality holds if and only if $\beta=0$ and $\sigma=\sigma_0$. Thus, $P^\beta_5$ is always beyond $P^\beta_0$ on $\Gamma_{wing}$ with $\beta>0$. This ensures the existence of the oblique shock $S^\beta_{ob}$; namely, the shock will attach to the leading edge rather than detach from it. 

As shown above, when the tangent point $P^\beta_1$ is above $P_2$, $\mathcal{C}^\beta_\sigma$ and $\mathcal{C}_\infty$ are only tangent at $P^\beta_1$. Then, thanks to the continuity of the position of $P^\beta_1$ with respect to $\beta$, we can calculate $\beta_c$ by solving the triangle $\bigtriangleup OP_2P^\beta_5$ when $P^\beta_1$ and $P_2$ coincide; that is, $S^\beta_{ob}$,  $\mathcal{C}^\beta_\sigma$ and $\mathcal{C}_\infty$ are tangent at the point $P_2$ (see Figure \ref{fgbetaeqc}). It is straightforward to derive $\beta_c = \arcsin(-\xi_{1P_2}/\cot\sigma)$. We note that when $\beta = \beta_c$, $v^\beta_{1\sigma} = v^\beta_{2\sigma} = 0$. It is then straightforward to verify the positivity of $\phi_\infty$ on $\mathcal{C}_\infty$ and $\phi^\beta_\sigma$ on $\mathcal{C}^\beta_\sigma$ for $\beta = \beta_c$.
\end{proof}

\begin{remark}\rm
    It follows from \eqref{cinf} that $\xi_{1P_2}$ is uniquely determined by the incoming flow. Hence, the explicit expression of $\beta_c$ in \eqref{eqdefinebetac} implies that for each $\sigma \in (0, \sigma_0]$, there exists a unique $\beta_c = \beta_c(\rho_\infty, q_\infty, \alpha, \sigma)$ determined by the incoming flow parameters and the sweep angle $\sigma$. In other words, for a given incoming flow and a planar attached shock, the corresponding conical wing with $\Lambda$-shaped cross sections is not unique.\qed
\end{remark}

\begin{remark}\rm 
    We briefly analyze the relative locations of $P^\beta_5,~P^\beta_0$ and $P^\beta_4$ on $\Gamma_{wing}$ when $\beta\in(0,\beta_c]$. As above, it follows from \eqref{p0p5} that $P^\beta_5$ is always beyond $P^\beta_0$ on $\Gamma_{wing}$. Besides, from the discussion in Lemma $\ref{lemma22}$, we see that $\mathcal{C^\beta_\sigma}$ is an inscribed cone of $\mathcal{C_\infty}$ and the tangent point is $P^\beta_1$, i.e., $|OP^\beta_0|>|OP^\beta_4|$. Hence, the $\xi_2$--coordinates of $P^\beta_i$ for $i=0,4,5$ satisfy $\xi_{2P^\beta_5}>\xi_{2P^\beta_0}>\xi_{2P^\beta_4}$. \qed
\end{remark}

Using the above analysis, we are able to draw the patterns of shock waves
as in Figures \ref{fgbetageq0} and \ref{fgbetaeqc} for the cases $0<\beta<\beta_c$ and $\beta=\beta_c$, respectively.

%-----------------------------------fig-----------
\begin{figure}[htbp]\label{betatobetac}
\centering
\begin{minipage}[t]{0.9\textwidth}
\centering
\subfigure[$0<\beta<\beta_c$]{
\includegraphics[width=5.8cm]{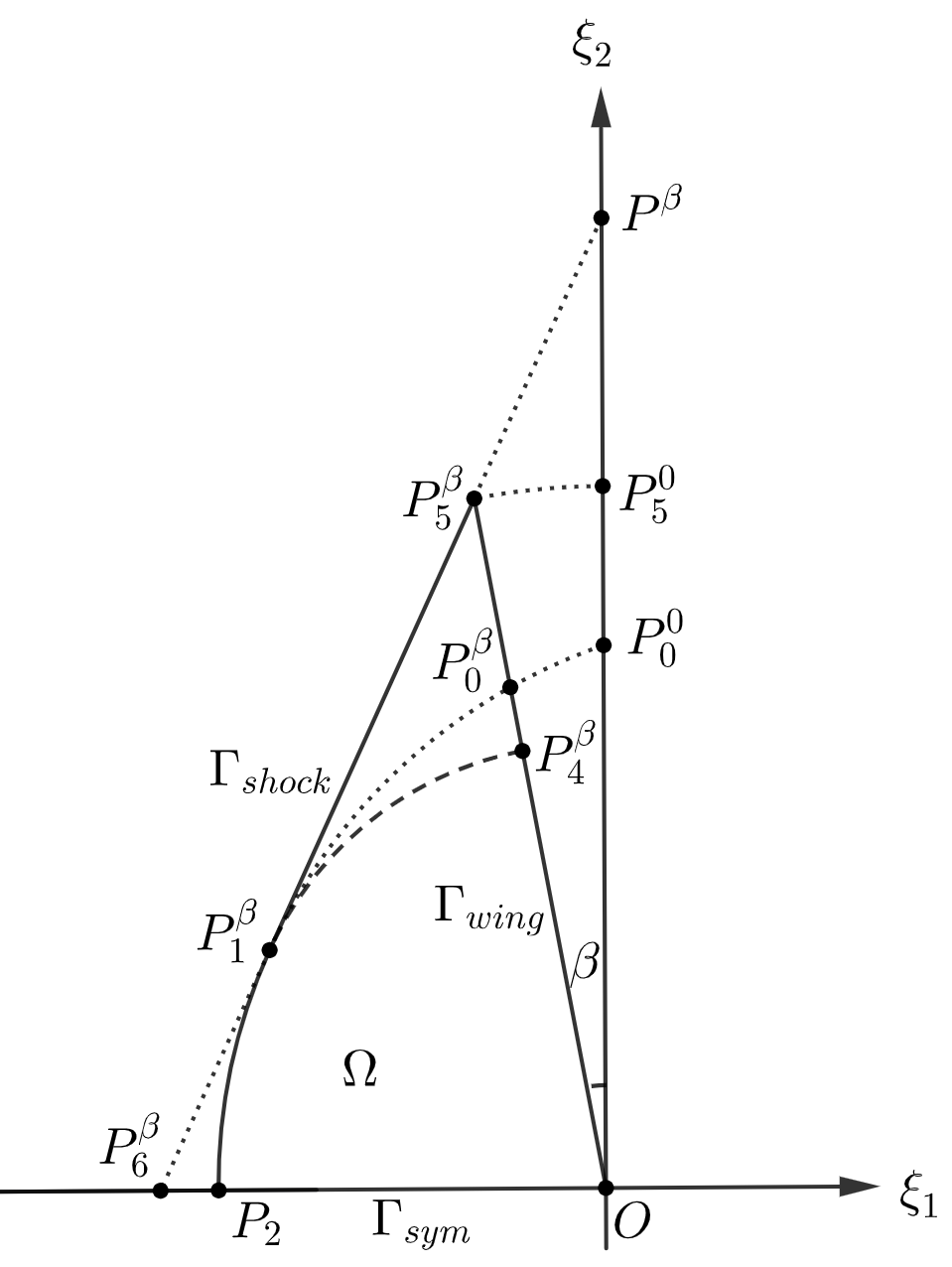}
\label{fgbetageq0}}
\subfigure[$\beta=\beta_c$]{
\includegraphics[width=6.4cm]{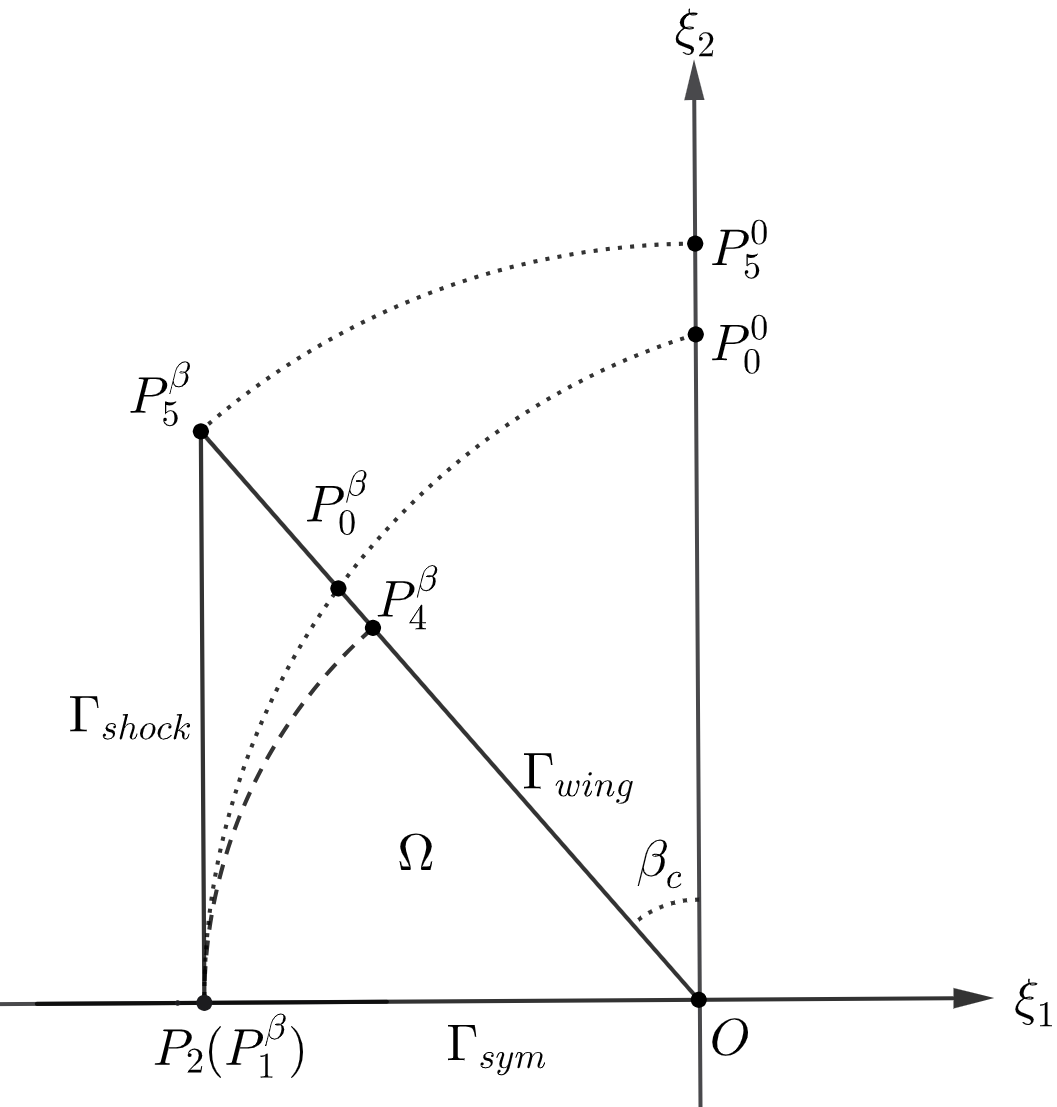}
\label{fgbetaeqc}}
\end{minipage}
\caption{Patterns of shock waves in the $(\xi_1,\xi_2)$-plane.}
\end{figure}
 %-----------------------------------fig-----------
 
We are in a position to consider the case $\beta>\beta_c$. Note that the oblique shock $S^\beta_{ob}$ is always tangent to the Mach cone $\mathcal{C}_\infty$ at $P^\beta_1$, and this tangent point will move to the lower half-plane of $\xi_1O\xi_2$ once $\beta>\beta_c$. Therefore, the attached shocks will intersect at a point $P^\beta_6$ on the $\xi_1$-axis before they are tangent to $\mathcal{C}_\infty$. In this paper, we only study the case in which two resulting shock waves are generated. Because of the symmetric patterns, we can still consider the upper half-plane of $\xi_1O\xi_2$. 

Let $S^\beta_R$ be the resulting oblique shock, and $\mathcal{C}^{\beta'}_\sigma$ be the curve corresponding to the Mach cone determined by the uniform flow behind $S^\beta_R$, and $P^\beta_R$ be the tangent point of $\mathcal{C}^\beta_\sigma,~\mathcal{C}^{\beta'}_\sigma$ and $S^\beta_R$. The state of uniform flow behind $S^\beta_R$ can be calculated as in Section \ref{sub2.1}, and the details are presented at the last of this subsection. Intuitively, this model seems to be reducible to a problem of two-dimensional steady shock wave regular reflection as in \cite[Section 3.5]{Serre09} (see Figure \ref{fgbetageqbetac}), but actually this is unreasonable; see Remark \ref{reserre} below for details.   

%-----------------------------------fig-----------
\begin{figure}[htbp]
\centering
\begin{minipage}[t]{0.9\textwidth}
\centering
\subfigure[$\beta_c<\beta<\beta_0$]{
\includegraphics[width=6.3cm]{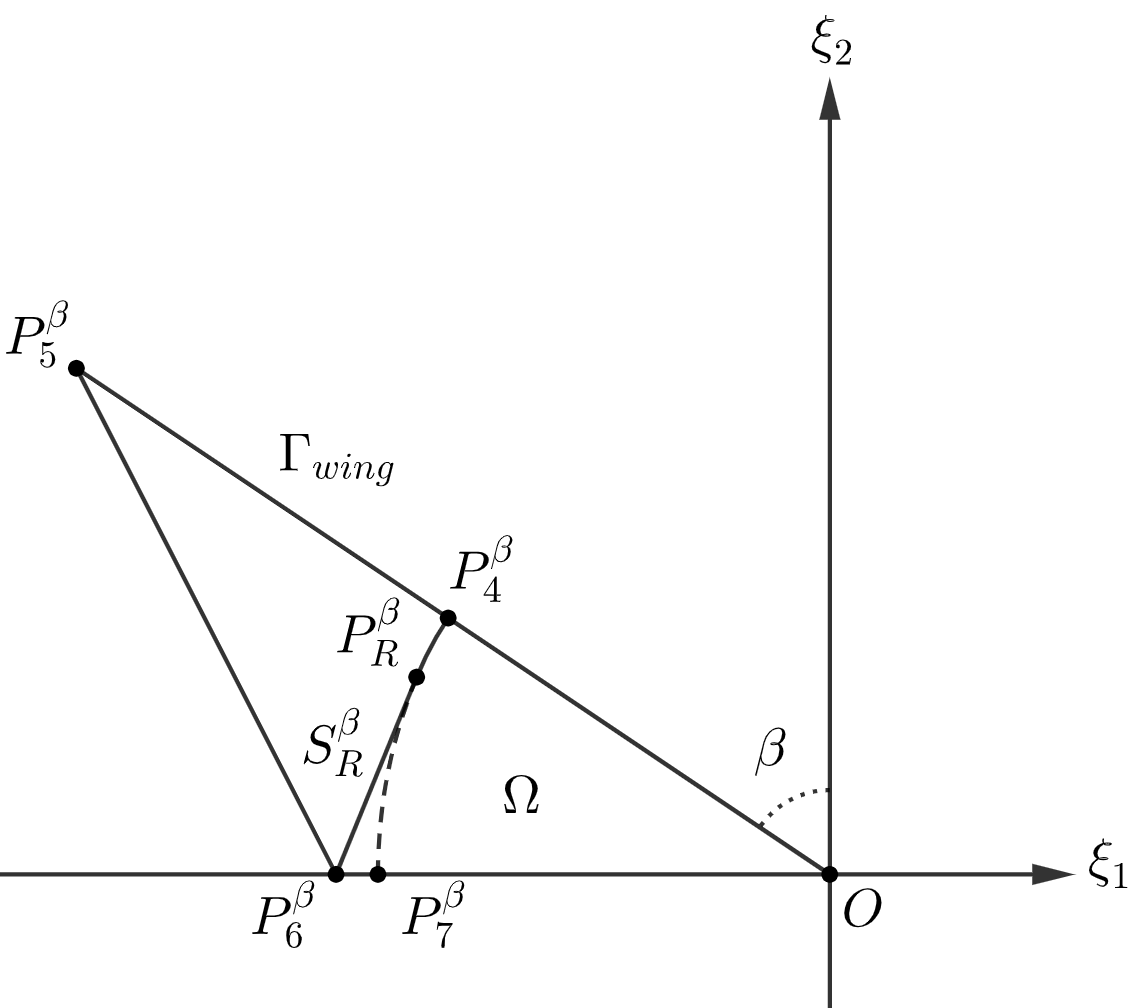}
\label{fgbetageqbetac}}
\subfigure[$\beta=\beta_0$]{
\includegraphics[width=6.5cm]{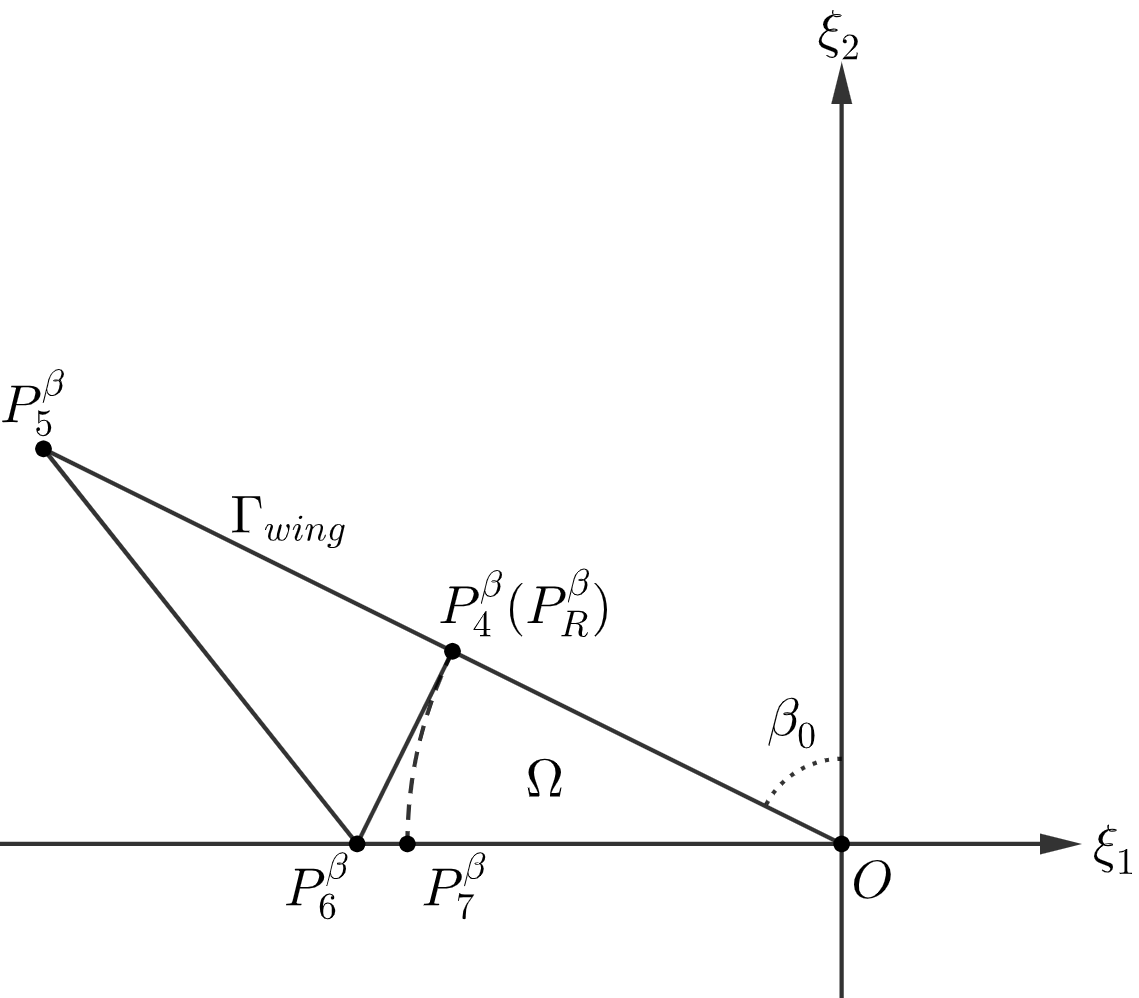}
\label{fgbetaeqbeta0}}
\end{minipage}
\caption{Patterns of shock waves in the $(\xi_1,\xi_2)$-plane.}\label{betact0beta0}
\end{figure}
%-----------------------------------fig-----------

\begin{lemma}
    Let $\mathcal{C}_\infty, \mathcal{C}^\beta_\sigma$ be defined as in \eqref{cinf}--\eqref{csigma} with $\beta>\beta_c$. For any fixed $\alpha \in(0, \alpha_0)$ and $\sigma \in (0, \sigma_0]$, there exists an angle $\beta_0$ satisfying  
    \begin{align}\label{debeta0}
    	\beta_0=\min_\beta \{\beta=\arcsin{\frac{|OP^\beta_4|}{|OP^\beta_6|}}\},
    \end{align}
     so that the flow is uniform for $\beta\in(\beta_c,\beta_0]$ in the domain $U\setminus \Omega$, where $P^\beta_7$ is the intersection point of $\mathcal{C^{\beta'}_\sigma}$ and $\xi_1$-axis, while $U$ and $\Omega$ respectively denote the domains $OP^\beta_5 P^\beta_6 $  and $OP^\beta_4 P^\beta_R P^\beta_7$ here, and $\alpha_0$, $\sigma_0$ are given by \eqref{alpha0}, \eqref{sigma0}, respectively.
\end{lemma}

\begin{proof}  
    Let us first give the necessary explanation for the geometric relations in the structures when $\beta>\beta_c$. Recall that $\mathcal{C}^\beta_\sigma$ is determined by the uniform flow behind $S^\beta_{ob}$, and this flow satisfies the slip boundary condition \eqref{bcwing} on the conical wing $\mathcal{W}^\beta_\sigma$. Thus, the curve $\mathcal{C}^\beta_\sigma$ is perpendicular to $\Gamma_{wing}$ at $P^\beta_4$ in conical coordinates. Also note that if there exists a critical angle $\beta_0$ such that $S^\beta_R$ exists and is perpendicular to $\Gamma_{wing}$, then the point $P^\beta_R$ must coincide with $P^\beta_4$ and $\angle P^\beta_4P^\beta_6O=\beta_0$ (cf. Figure \ref{fgbetaeqbeta0}). Hence, the relation $\angle P^\beta_4P^\beta_6O>\beta$ is equivalent to the tangent point $P^\beta_R$ of $S^\beta_R$ and $\mathcal{C}^\beta_\sigma$ lying inside the $\Delta OP^\beta_5P^\beta_6$. Note that, as mentioned in Remark \ref{reserre}, it is difficult to calculate the variation of angle $\angle P^\beta_4P^\beta_6O$ directly, we thus analyze the critical cases with the continuous variation of $\angle P^\beta_4P^\beta_6O$ with respect to $\beta$.
 
    We then consider the variation of angle $\angle P^\beta_4P^\beta_6O$ as $\beta$ increases from $\beta_c$ to $\pi/2$. For the critical case $\beta=\beta_c$, $P^\beta_6$ coincides with $P^\beta_1$ and $P_2$ so that the line segment $P^\beta_6P^\beta_4$ inside $\mathcal{C}^\beta_\sigma$ makes $\angle P^\beta_4P^\beta_6O>\beta$. In the other limit case $\beta\to\pi/2$, we notice that $\angle P^\beta_4P^\beta_6O\to 0$ with $\angle P^\beta_4P^\beta_6O<\beta$. Thus, there must be at least a $\beta\in(\beta_c,\pi/2)$ so that $\angle P^\beta_4P^\beta_6O=\beta$ since $\angle P^\beta_4P^\beta_6O$ is a continuous function of $\beta$, where $\beta=\arcsin(|OP^\beta_4|/|OP^\beta_6|)$. We denote by $\beta_0$ the minimum of these angles; that is,
    $\beta_0= \min\{\beta~|~\beta=\arcsin(|OP^\beta_4|/|OP^\beta_6|)\}$, so that $\angle P^\beta_4P^\beta_6O>\beta$ when $\beta\in(\beta_c,\beta_0)$ for the continuity with respect to $\beta$. The proof is complete. 
\end{proof}

\begin{remark}\label{reserre}\rm
   To our knowledge, for the Chaplygin gas, the global flow field structures of previously studied problems can all be reduced to two dimensions and then obtained by the relationship that oblique shock waves are tangent to the sonic circles determined by the states on both sides of them; see \cite{CQ12SCM,CQu12,CCW22,Serre09} for the two-dimensional Riemann problem and the problem of two-dimensional shock wave reflection. In our study, although the problem has a self-similar structure and can thus be reduced to two dimensions, it is impossible to find a two-dimensional plane such that the corresponding curves of Mach cones $\mathcal{C}_\infty,~\mathcal{C}^\beta_\sigma$ and $\mathcal{C}^{\beta'}_\sigma$ on it are all circles. Thus, the geometric method developed in \cite{Serre09} by Serre is not available here, and the uniqueness of $\beta_0$ is hard to obtain although intuitively it is correct.\qed
\end{remark}

We finally conclude Sections \ref{sub2.1}--\ref{sec2.2} by the following proposition.
\begin{proposition}\label{uniform flow}
Assume that the incoming flow is uniform and supersonic with state \(U_\infty=(\rho_{\infty},q_{\infty})\), and that the conical wing \(\mathcal{W}^\beta_\sigma\) is given by \eqref{sweep wing}. Then there exists a critical attack angle \(\alpha_0 = \alpha_0(\rho_\infty, q_\infty) \in (0, \pi/2)\) such that for any fixed \(\alpha \in (0, \alpha_0)\), there exists a critical sweep angle \(\sigma_0 = \sigma_0(\rho_\infty, q_\infty, \alpha) \in (0, \pi/2)\) such that for any fixed \(\sigma \in (0, \sigma_0)\), there exists a critical anhedral angle \(\beta_0 = \beta_0(\rho_\infty, q_\infty, \alpha, \sigma) \in (0, \pi/2)\) with the property that for all \(\beta \in (0, \beta_0]\), the flow is uniform in the domain \(U\setminus\Omega\) (see Figures $\ref{fgbetageq0}$--$\ref{fgbetaeqc}$ and $\ref{fgbetageqbetac}$--$\ref{fgbetaeqbeta0}$). Here \(\alpha_0\) and \(\sigma_0\) are defined by \eqref{alpha0} and \eqref{sigma0}, respectively.
\end{proposition}

\begin{proof}    
We first recall that for any fixed $\alpha \in(0, \alpha_0)$ and $\sigma \in (0, \sigma_0]$, $U$ denotes the domain $OP_2P^\beta_1P^\beta_5$, and $\Omega$ denotes the domain $OP_2P^\beta_1P^\beta_4$ when $\beta\in(0,\beta_c]$; while $U$ denotes the domain $OP^\beta_5P^\beta_6$, and $\Omega$ denotes the domain $OP^\beta_4 P^\beta_R P^\beta_7$ when $\beta\in(\beta_c,\beta_0]$. Moreover, the global structures of shocks in conical coordinates have been established; see Figures \ref{fgbetageq0} and \ref{fgbetaeqc} for $\beta\in(0,\beta_c]$, and Figures \ref{fgbetageqbetac} and \ref{fgbetaeqbeta0} for $\beta\in(\beta_c,\beta_0]$.
    
Note that, for the case $\beta\in(0,\beta_c]$, the uniform flow behind the oblique shock $S^\beta_{ob}$ is calculated in Section \ref{sub2.1}. Then, it remains to calculate the state of uniform flow behind $S^\beta_R$ (i.e., the domain $P^\beta_6 P^\beta_7 P^\beta_R$ in Figures \ref{fgbetageqbetac} and \ref{fgbetaeqbeta0}) to complete the proof.

For the symmetry of shock patterns, the velocity of uniform flow behind $S^\beta_R$ can be denoted by
\begin{align*}
    \bm{v}^\beta_R=v^\beta_{1R}\bm{e}_1+v^\beta_{3R}\bm{e}_3.
\end{align*}
Then we consider the intersection line of $S^\beta_{ob}$ and $x_1Ox_3$-plane denoted by $L_R$,
\begin{align*}
    L_R:~ x_1+\frac{v^\beta_{3\sigma}-v_{3\infty}}{v^\beta_{1\sigma}-v_{1\infty}}x_3=0.
\end{align*}
As the analysis in Section \ref{sub2.1}, we decompose $\bm{v}^\sigma_\beta$ in the $x_1Ox_3$-plane with $L_R$ as the new leading edge. Let us introduce a new orthogonal basis:
\begin{align*}
		&\bm{e}'_i\doteq(0,-1, 0), \\
        &\bm{e}'_j\doteq(-\frac{v^\beta_{3\sigma}-v_{3\infty}}{v^\beta_{1\sigma}-v_{1\infty}}, 0, 1), \\&\bm{e}'_k\doteq(\frac{v^\beta_{1\sigma}-v_{1\infty}}{v^\beta_{3\sigma}-v_{3\infty}}, 0, 1).
	\end{align*}
Clearly, $\bm{v}^\beta_\sigma$ can be decomposed as
\begin{align*}
    \bm{v}^\beta_\sigma=-v^\beta_{2\sigma}\bm{e}'_i+(v^\beta_{3\sigma}-v^\beta_{1\sigma}\frac{v^\beta_{3\sigma}-v_{3\infty}}{v^\beta_{1\sigma}-v_{1\infty}})\bm{e}'_j+(v^\beta_{3\sigma}+v^\beta_{1\sigma}\frac{v^\beta_{1\sigma}-v_{1\infty}}{v^\beta_{3\sigma}-v_{3\infty}})\bm{e}'_k.
\end{align*}
Corresponding to Section \ref{sub2.1}, we set
\begin{align*}
    \tilde{\bm{v}}^\beta_\sigma=-v^\beta_{2\sigma}\bm{e}'_i+(v^\beta_{3\sigma}+v^\beta_{1\sigma}\frac{v^\beta_{1\sigma}-v_{1\infty}}{v^\beta_{3\sigma}-v_{3\infty}})\bm{e}'_k
\end{align*}
to denote the velocity component of the flow perpendicular to $L_R$ with
\begin{align*}
    \tilde{q}^\beta_\sigma=\sqrt{(-v^\beta_{2\sigma})^2+(v^\beta_{3\sigma}+v^\beta_{1\sigma}\frac{v^\beta_{1\sigma}-v_{1\infty}}{v^\beta_{3\sigma}-v_{3\infty}})^2},
\end{align*}
and write the angle between $\tilde{\bm{v}}^\beta_\sigma$ and $\bm{e}'_k$ as
\begin{align*}
    \theta'_n=\arctan\left(\frac{-v^\beta_{2\sigma}(v^\beta_{3\sigma}-v_{3\infty})}{v^\beta_{3\sigma}(v^\beta_{3\sigma}-v_{3\infty})+v^\beta_{1\sigma}(v^\beta_{1\sigma}-v_{1\infty})}\right).
\end{align*}
Let $q^\beta _{jR}$ be the speed of uniform flow behind $S^\beta_R$ along $\bm{e}'_k$. Moreover, noting that $\xi_{1P^\beta_6}=-(v^\beta_{3\sigma}-v_{3\infty})/(v^\beta_{1\sigma}-v_{1\infty})$, we get the uniform flow with 
\begin{align*}
    &v^\beta_{1R}=v^\beta_{3\sigma}\xi_{1P^\beta_6}+v^\beta_{1\sigma}\xi^2_{1P^\beta_6}-\frac{1}{\xi_{1P^\beta_6}}q^\beta_{jR},\\
  & v^\beta_{3R}=v^\beta_{3\sigma}+v^\beta_{1\sigma}\xi_{1P^\beta_6}+q^\beta_{jR},
\end{align*}
 where $q^\beta_{jR}$ can be derived from \eqref{u1v1polar} with the choice $u_0 = \tilde{q}^\beta_\sigma, c_0 = c^\beta_\sigma$ and $\theta = \theta'_n$. Thus, we get
\begin{align*}   S^\beta_R:~v^\beta_{1R}\xi_1+v^\beta_{3R}=v^\beta_{1\sigma}\xi_1+v^\beta_{2\sigma}\xi_2+v^\beta_{3\sigma},
\end{align*}
and
\begin{align*}
    \mathcal{C}^{\beta'}_\sigma:~~(v^\beta_{1R}\xi_1+v^\beta_{3R})^2=1+|\bm{\xi}|^2
\end{align*}
corresponding to the Mach cone formed by the uniform flow behind $S^\beta_R$.
\end{proof}
 
 \subsection{Simplification of Problem A}
In this subsection, to rigorously reduce the initial Problem A into a boundary value problem in the conical coordinates, we first analyze the governing equation. Introducing the following notation
\begin{align*}
\mathrm{D}^2\varphi[\bm{a},\bm{b}]\doteq\sum^2_{i,j=1}a_ib_j\partial_{ij}\varphi\quad \text{for} ~\varphi\in C^2~\text{and}~\bm{a},\bm{b}\in \mathbb{R}^2,
\end{align*}
and utilizing \eqref{second order} and \eqref{scaling}, the equation for $\phi$ takes the form 
\begin{align}\label{D2conicalpotential}
    c^2(\Delta\phi+\rmD^2\phi[\bm{\xi},\bm{\xi}])-\rmD^2\phi[\rmD\phi-(\phi-\rmD\phi\cdot\bm{\xi})\bm{\xi}, \rmD\phi-(\phi-\rmD\phi\cdot\bm{\xi})\bm{\xi}]=0,
\end{align}
where $\Delta$ and $\rmD$ denote the Laplacian and the gradient operator with respect to $\bm{\xi}$.

The classification of equation \eqref{D2conicalpotential} is governed by
\begin{align}\label{L2}
    L^2\doteq\frac{|\rmD\phi|^2+|\phi-\rmD\phi\cdot\bm{\xi}|^2-\frac{\phi^2}{1+|\bm{\xi}|^2}}{c^2},
\end{align}
with $c^2=|\mathrm{D}\phi|^2+|\phi-\mathrm{D}\phi\cdot\bm{\xi}|^2-1$. Since equation \eqref{D2conicalpotential} is hyperbolic for $L^2 > 1$, parabolic degenerate for $L^2 = 1$, and elliptic for $L^2 < 1$, the expression \eqref{L2} implies that it is hyperbolic in the region $U\setminus \Omega$ and parabolic degenerate on the boundary $\Gamma_{cone}^{\infty}\cup\Gamma_{cone}^{\beta}$.

We remark here that as $\beta$ varies, the domains corresponding to $U$ and $\Omega$, along with the degenerate curve corresponding to $\Gamma_{cone}^{\infty}\cup\Gamma_{cone}^{\beta}$, also change (see Section \ref{sec2.2}). By abuse of notation but without misunderstanding, throughout this paper, we consistently denote $\Omega$ as the elliptic domain corresponding to disturbed flow, $U\setminus \Omega$ as the hyperbolic domain corresponding to uniform flow, and $\Gamma_{cone}^{\infty}\cup\Gamma_{cone}^{\beta}$ as the degenerate curve separating them.

We now proceed to consider the boundary conditions. In view of the positivity of potential functions established in Lemma $\ref{lemma22}$, the boundary condition for parabolic degeneracy on $\Gamma_{cone}^{\infty}\cup\Gamma_{cone}^{\beta}$ is equivalent to imposing the Dirichlet condition 
\begin{align}\label{phicon}
    \phi=\sqrt{1+|\bm{\xi}|^2}.
\end{align}
Moreover, in the conical coordinates, the slip conditions \eqref{bcwing}--\eqref{bcsym} can be reformulated as:
\begin{align}
\rmD\phi\cdot\bm{\nu}_w=0\quad &\text{on}~\Gamma_{wing},\label{conicalbcwing}\\
\rmD\phi\cdot\bm{\nu}_{sy}=0\quad &\text{on}~\Gamma_{sym},\label{connialbcsym}
\end{align}
where $\bm{\nu}_w=(\cos\beta,\sin\beta)$ and $\bm{\nu}_{sy}=(0,-1)$ are the exterior unit normals to $\Gamma_{wing}$ and $\Gamma_{sym}$, respectively. 

Accordingly, combined with Proposition \ref{uniform flow}, Problem A can be simplified as follows in the conical coordinates.
%%%%%%%%%%%%%%%%%%%%%%%%%%%%%%%%%%%%
\vspace{\baselineskip} % 段前间距    
\begin{center}    
\fbox{\begin{varwidth}{0.9\linewidth} % 设置宽度为行宽的80%        
\centering \textbf{Problem B}: For the wing $\mathcal{W}^\beta_\sigma$ and the incoming flow $U_{\infty}$ given by (\ref{sweep wing}) and (\ref{incoming flow}) respectively, find a solution to \eqref{D2conicalpotential} in the region $\Omega$ with the Dirichlet condition (\ref{phicon}) on $\Gamma_{cone}^{\infty}\cup\Gamma_{cone}^{\beta}$ and slip conditions \eqref{conicalbcwing}--\eqref{connialbcsym}. % 内容居中    
\end{varwidth}}    
\end{center}    
\vspace{\baselineskip} % 段后间距
%%%%%%%%%%%%%%%%%%%%%%%%%%%%%%%%%%%%% 

Henceforth, to show Theorem \ref{thm: Main}, it suffices to prove the following theorem. 
\begin{theorem}\label{thm4}
Let the angles \(\alpha_{0},\sigma_{0},\alpha,\sigma\) and \(\beta_{0}\) be defined as in Proposition \(\ref{uniform flow}\). Then, for any fixed \(\beta\in(0,\beta_{0}]\), Problem B admits a solution \(\phi\) with the following regularity:
\begin{equation*}
\phi\in \mathrm{Lip}(\overline{\Omega})\cap C^{1,\kappa}(\overline{\Omega}\setminus\overline{\Gamma_{cone}^{\infty}\cup\Gamma_{cone}^{\beta}})\cap C^{\infty}(\Omega\cup\Gamma_{sym}\cup\Gamma_{wing}),
\end{equation*}
where \(\kappa=\kappa(\alpha,\sigma,\beta)\in(0,1)\) is a constant. Moreover, \(\phi\) satisfies the ellipticity condition
\begin{equation}\label{ellipticcondition}
\phi >\sqrt{1+|\bm{\xi}|^{2}} \quad\text{in } \overline{\Omega}\setminus\overline{\Gamma_{cone}^{\infty}\cup\Gamma_{cone}^{\beta}}.
\end{equation}
\end{theorem}

\begin{remark}\rm 
As mentioned above, equation \eqref{D2conicalpotential} is hyperbolic in the domain $U\setminus \Omega$ and parabolic degenerate on the arc $\Gamma_{cone}^{\infty}\cup\Gamma_{cone}^{\beta}$. Moreover, it is shown in \cite[Lemma A.1]{L25} that \eqref{D2conicalpotential} is elliptic in the interior of a parabolic-elliptic region. Thus, it is natural to assume that equation \eqref{D2conicalpotential} is elliptic in $\overline{\Omega}\setminus\overline{\Gamma_{cone}^{\infty}\cup\Gamma_{cone}^{\beta}}$, i.e., \eqref{ellipticcondition}. \qed
\end{remark}

\begin{remark}\rm 
Due to the existence of the Neumann boundary conditions, different from \cite{LY22,Serre09}, we cannot obtain the uniqueness of the solution to Problem B by applying the comparison principle established in Lemma $\ref{lemma: comparison principle}$ below. \qed 
\end{remark}

\begin{remark}\label{differentLY22} \rm 
Note that the works in \cite{LY22,Serre09} mainly focus on the Dirichlet problem. To prove Theorem $\ref{thm4}$, owing to the anhedral angle $\beta>0$, we also need to deal with the corner point and Neumann boundary conditions. These difficulties are also briefly addressed in \cite{CQ12SCM,L25,LY22}. However, in this paper, the establishment of the $L^\infty$-estimate \eqref{eq: estimate phi} for the viscosity solution below is most crucial and different. \qed 
\end{remark}

\section{Analysis of the flow inside Mach cone}\label{sec:3}

We now present the proof strategy of Theorem \ref{thm4}. Introduce the following family of boundary-value problems:
	\begin{align}\label{eq: for Fmu}
	\sum^{2}_{i,j=1} A_{ij}(\mu, \phi) \partial_{ij}\phi\doteq c^{2}(\Delta\phi+\rmD^{2}\phi[\bm{\xi},\bm{\xi}])& \nonumber\\-\mu \rmD^{2}\phi[\rmD\phi-(\phi-\rmD\phi\cdot\bm{\xi})\bm{\xi}, &\rmD\phi-(\phi-\rmD\phi\cdot\bm{\xi})\bm{\xi}]=0\quad\text{in $\Omega$},
	\end{align}
where $\partial_{ij}\phi\doteq\partial_{\xi_i\xi_j}\phi$ with $i,j\in\{1,2\}$, and
\begin{equation}\label{eq:nianxing BC}
	\begin{cases}
		\phi=\sqrt{1+|\bm{\xi}|^2}+\varepsilon\quad&\text{on $\Gamma_{cone}^{\infty}\cup\Gamma_{cone}^{\beta}$},\\
		\rmD\phi\cdot\bm{\nu}_{w}=0\quad&\text{on $\Gamma_{wing}$},\\
		\rmD\phi\cdot\bm{\nu}_{sy}=0\quad&\text{on $\Gamma_{sym}$}.
	\end{cases}
\end{equation}
Here, $\mu\in[0,1]$ and $\varepsilon>0$ are two parameters. Let $\phi_{\mu,\varepsilon}$ be the solution to problem \eqref{eq: for Fmu}--\eqref{eq:nianxing BC}. It can be directly verified that for all $\mu\in[0,1]$, equation \eqref{eq: for Fmu} is elliptic if $\phi_{\mu,\varepsilon}>\sqrt{1+|\bm{\xi}|^2}$, and it is uniformly elliptic if
\begin{align}
	\sqrt{1+|\bm{\xi}|^2}+\delta_{0}\leq  \phi_{\mu,\varepsilon}&<C,\label{eq: estimate phi}\\
	\vert \rmD\phi_{\mu,\varepsilon}\vert&<C,\label{eq: lip estimate phi}
\end{align}
where $\delta_{0},C$ are two positive bounded constants.

In view of the above property of equation \eqref{eq: for Fmu}, we add the viscosity parameter $\varepsilon$ to the degenerate boundary condition on $\Gamma_{cone}^{\infty}\cup\Gamma_{cone}^{\beta}$ as in \eqref{eq:nianxing BC}\(_1\), which is inspired by the technique in \cite{Serre09}. In addition, the parameter $\mu$ plays a key role: equation \eqref{eq: for Fmu} is linear and uniformly elliptic at $\mu = 0$; it returns to the original nonlinear equation \eqref{D2conicalpotential} at $\mu = 1$. 

Define a set
\begin{equation*}
	\begin{aligned}
		J_{\varepsilon}\doteq\{\mu\in[0,1]:~&\phi_{\mu,\varepsilon}\in C^{1}(\overline{\Omega})\cap C^2(\Omega\cup\Gamma_{sym}\cup\Gamma_{wing}) ~\text{satisfies}~\eqref{eq: for Fmu}\text{-}\eqref{eq:nianxing BC}\\		&~\text{with}~\phi_{\mu,\varepsilon}\geq \sqrt{1+\vert\bm{\xi}\vert^2}+\varepsilon~\text{in}~\overline{\Omega}\setminus\overline{\Gamma_{cone}^{\infty}\cup\Gamma_{cone}^{\beta}}\}.
	\end{aligned}
\end{equation*}
The strategy for the proof of Theorem \ref{thm4} can be divided into two steps. To be specific, for any fixed $\varepsilon>0$, we first apply the continuity method to show the existence of the viscosity solution to the problem \eqref{D2conicalpotential} with \eqref{eq:nianxing BC}; namely, $J_{\varepsilon}=[0,1]$. A key ingredient here is establishing the uniform estimates \eqref{eq: estimate phi}--\eqref{eq: lip estimate phi} that are independent of $\mu$ and $\varepsilon$, to be demonstrated in Section \ref{sec: prior estimate}. Second, we prove that the viscosity solution $\phi_{1,\varepsilon}$ converges to a solution of Problem B as $\varepsilon \to 0^+$.

\subsection{Lipschitz estimate}\label{sec: prior estimate}
It is difficult to obtain the $L^\infty$ estimate of $\phi_{\mu,\varepsilon}$ directly because we can not establish directly a comparison principle for the equation \eqref{eq: for Fmu}. To overcome this difficulty, we introduce an auxiliary function used in \cite{LY22}: 
\begin{equation}\label{eq phi to w}
	w_{\mu,\varepsilon}\doteq \frac{\phi_{\mu,\varepsilon}}{\sqrt{1+|\bm{\xi}|^{2}}}.
\end{equation}

Due to the bounded domain $\Omega$, there is no need to differentiate $\varepsilon$ from $\sqrt{1+|\bm{\xi}|^2} \varepsilon$, and we simply write $\varepsilon$ for both in the subsequent analysis.

By \eqref{eq: for Fmu}--\eqref{eq:nianxing BC} and \eqref{eq phi to w}, the function $w_{\mu,\varepsilon}$ satisfies
\begin{multline}\label{eq for w}
	c^{2}(\Delta{w}+\rmD^{2}{w}[\bm{\xi},\bm{\xi}])-\mu(1+|\bm{\xi}|^{2})\rmD^{2}{w}[\rmD{w}+(\rmD{w}\cdot\bm{\xi})\bm{\xi}, \rmD{w}+(\rmD{w}\cdot\bm{\xi})\bm{\xi}]\\
	+2\left((1-\mu)c^{2}+\mu(w^2-1)\right)\rmD{w}\cdot\bm{\xi}+\frac{\left((2-\mu)c^2+\mu(w^2-1)\right)w}{1+|\bm{\xi}|^{2}}=0,
\end{multline}
with
\begin{equation}\label{BC for w}
	\begin{cases}
		w=1+\varepsilon\quad&\text{on $\Gamma_{cone}^{\infty}\cup\Gamma_{cone}^{\beta}$},\\
		\rmD w\cdot\bm{\nu}_{w}=0\quad&\text{on $\Gamma_{wing}$},\\
		\rmD w\cdot\bm{\nu}_{sy}=0\quad&\text{on $\Gamma_{sym}$}.
	\end{cases}
\end{equation}
Denote by $\mathcal{N}_\mu w$ the left-hand side of \eqref{eq for w}. We have the following comparison principle.

\begin{lemma}\label{lemma: comparison principle}
Let $\Omega_{d}\subset \mathbb{R}^{2}$ be a bounded open domain whose boundary $\partial\Omega_{d}$ consists of finitely many smooth arcs $\Gamma_i,\dots,\Gamma_n$. Assume that at any intersection point of $\Gamma_k$ and $\Gamma_l$, the angle formed lies in $(0,\pi]$. Consider functions ${w}_{\pm}$ belonging to $ C^{0}(\overline{\Omega_{d}})\cap C^{1}(\overline{\Omega_{d}}\setminus\overline{\Gamma_1\cup\Gamma_2})\cap C^2(\Omega_{d})$
and satisfying ${w}_{\pm}>1$ in $\Omega_{d}$. Suppose that for each $\mu\in [0,1]$, the operator $\mathcal{N}_\mu$ is locally uniformly elliptic with respect to either ${w}_{+}$ or $w_-$, and
	\begin{align*}
	\mathcal{N}_\mu w_-\geq0, \quad\mathcal{N}_\mu w_+\leq0\quad&\text{in}~\Omega_{d},\\
   w_-\leq w_+\quad&\text{on}~\Gamma_1\cup\Gamma_2,\\
     \rmD w_-\cdot\bm{\nu}_{n_j}<\rmD w_+\cdot\bm{\nu}_{n_j}\quad&\text{on}~\Gamma_j~\text{for}~j=3,\dots,n,
	\end{align*}
with $\bm{\nu}_{n_j}$ being the exterior unit normal to $\Gamma_j$. Then, ${w}_{-}\leq{w}_{+}$ holds in $\overline{\Omega_{d}}\setminus\overline{\Gamma_1\cup\Gamma_2}$.
\end{lemma}

\begin{proof}
	Set
	\begin{equation*}
		\bar{w}\doteq w_- -w_+.
	\end{equation*}
Then, we obtain
	\begin{equation*}
    \sup_{\Omega_d}\bar{w}\leq\sup_{\partial\Omega_d} \bar{w}\leq\sup_{\Gamma_1\cup\Gamma_2}\bar{w},
	\end{equation*}
where the first inequality follows from \cite[Lemma 3.1]{LY22}, and the second one is obtained by the condition $\rmD\bar{w}\cdot\bm{\nu}_{n_j}<0$ on $\Gamma_j$. The proof is complete.
\end{proof}

We now use the auxiliary function $w_{\mu,\varepsilon}$ and the above comparison principle to obtain the estimates \eqref{eq: estimate phi}--\eqref{eq: lip estimate phi}. 
\begin{lemma}\label{lemma: lip}
Assume that $\phi_{\mu,\varepsilon} \in C^{0}(\overline{\Omega})\cap C^{1}(\overline{\Omega}\setminus\overline{\Gamma_{cone}^{\infty}\cup\Gamma_{cone}^{\beta}})\cap C^2(\Omega\cup\Gamma_{sym}\cup\Gamma_{wing})$ is a solution to \eqref{eq: for Fmu}--\eqref{eq:nianxing BC} with $\phi_{\mu,\varepsilon}\geq \sqrt{1+|\bm{\xi}|^2}+\varepsilon$ in $\overline{\Omega}\setminus\overline{\Gamma_{cone}^{\infty}\cup\Gamma_{cone}^{\beta}}$. Then, for any $\varepsilon>0$ and $\mu\in[0,1]$, one can find a constant $C$ depending neither on $\mu$ nor $\varepsilon$ such that
    \begin{align}
	\sqrt{1+|\bm{\xi}|^2}+\varepsilon<\phi_{\mu,\varepsilon}&\leq C\quad\text{in}~\overline{\Omega}\setminus\overline{\Gamma_{cone}^{\infty}\cup\Gamma_{cone}^{\beta}},\label{eq phi inf}\\
	\|\rmD{\phi}_{\mu,\varepsilon}\|_{L^{\infty}(\Omega)}&\leq C.\label{eq for phi lip}
\end{align}
\end{lemma}

\begin{proof}
To simplify the notation, we shall drop the subscripts of $\phi_{\mu,\varepsilon}$ and $w_{\mu,\varepsilon}$ in the proof of this lemma. The proof will be divided into three steps:

1. The auxiliary function $w$ given by \eqref{eq phi to w} enables us to construct a super-solution to the problem \eqref{eq: for Fmu} with \eqref{eq:nianxing BC}. Note that for an arbitrary constant vector $\bm{\eta}=(\eta_1,\eta_2,\eta_3)\in \mathbb{R}^{3}$,
    \begin{equation}\label{eqexactsolution}
		{w}^{\bm{\eta}}(\bm{\xi})=\frac{\bm{\eta}\cdot(\bm{\xi},1)}{\sqrt{1+|\bm{\xi}|^{2}}}
	\end{equation}
solves equation \eqref{eq for w} exactly, and this solution is Lipschitz bounded in $\Omega$. 

Let us consider a set with fixed $\varepsilon>0$
	\begin{equation*}
		\Sigma_{+}^{\beta}\doteq\{\bm{\eta}\in\mathbb{R}^{3}:~ P^{\hat{\bm{\eta}}}\in \mathcal{D}_1, \text{ and }{w}^{\bm{\eta}}>1+\varepsilon~\text{on the Dirichlet boundary}~\Gamma_{cone}^{\infty}\cup\Gamma_{cone}^{\beta}\},
	\end{equation*}
where $P^{\hat{\bm{\eta}}}=(\hat{\eta}_1,\hat{\eta}_2)=(\frac{\eta_1}{\eta_3},\frac{\eta_2}{\eta_3})$ and
    \begin{equation*}
    \mathcal{D}_1\doteq\{(\xi_1,\xi_2)~|~\frac{3\pi}{2}+\beta<\arctan(\xi_2/\xi_1)<2\pi\}.  
    \end{equation*}
The expression \eqref{eqexactsolution} implies that ${w}^{\bm{\eta}}$ is monotonically decreasing with respect to the angle between $\frac{(\bm{\xi},1)}{\sqrt{1+|\bm{\xi}|^{2}}}$ and $\bm{\eta}$ for each fixed $\bm{\eta}\in \mathbb{R}^{3}$. This monotonicity property ensures that for $\bm{\eta}\in \Sigma_{+}^{\beta}$, we have $\rmD w^{\bm{\eta}}\cdot\bm{\nu}_{w}>0$ and $\rmD w^{\bm{\eta}}\cdot\bm{\nu}_{sy}>0$ on the Neumann boundaries $\Gamma_{wing}\cup\Gamma_{sym}$, and moreover, the operator $\mathcal{N}_\mu $ is locally uniformly elliptic with respect to $w^{\bm{\eta}}$. Then, it follows from Lemma \ref{lemma: comparison principle} that $w\leq w^{\bm{\eta}}$ in $\overline{\Omega}\setminus\overline{\Gamma_{cone}^{\infty}\cup\Gamma_{cone}^{\beta}}$. We define $w^+$ as the infimum of all such supersolutions $w^{\bm{\eta}}$. Obviously, 
	\begin{equation}\label{sup for w}
		w\leq w^+\quad\text{in $\overline{\Omega}\setminus\overline{\Gamma_{cone}^{\infty}\cup\Gamma_{cone}^{\beta}}$}.
	\end{equation}
Furthermore, the convexity of the boundary $\Gamma_{cone}^{\infty}\cup\Gamma_{cone}^{\beta}$ ensures 
	\begin{equation}\label{w positive bd}
		w^+=1+\varepsilon=w\quad\text{on $\Gamma_{cone}^{\infty}\cup\Gamma_{cone}^{\beta}$}.
	\end{equation} 

2. With the use of \eqref{sup for w}--\eqref{w positive bd}, we then show the estimate \eqref{eq for phi lip}. By \cite[Lemma 3.5]{LY22}, the maximum of $|\rmD\phi|^{2}$ is attained at an interior point only if the function is constant, which leads to
	\begin{equation}\label{eq lip for w}
		\|\rmD{\phi}\|_{L^{\infty}(\Omega)}\leq \|\rmD{\phi}\|_{L^{\infty}(\partial\Omega)}.
	\end{equation}
Besides, the boundary conditions \eqref{eq:nianxing BC}$_2$--\eqref{eq:nianxing BC}$_3$ imply $\rmD{\phi}=0$ at the corner point $O$. Therefore, from Remark \ref{remark: property} below, we can reduce \eqref{eq lip for w} to
	\begin{equation}\label{eq lip for phi2}
		\|\rmD{\phi}\|_{L^{\infty}(\Omega)}\leq \|\rmD{\phi}\|_{L^{\infty}(\Gamma_{cone}^{\infty}\cup\Gamma_{cone}^{\beta})}\leq \|\rmD{\phi}^{+}\|_{L^{\infty}(\Gamma_{cone}^{\infty}\cup\Gamma_{cone}^{\beta})},
	\end{equation}
where $\phi^+\doteq\sqrt{1+|\bm{\xi}|^2}w^+$, and the second inequality is derived from \eqref{sup for w}--\eqref{w positive bd}. Thus, the estimate \eqref{eq for phi lip} holds.

3. We finally establish a sub-solution to the problem \eqref{eq: for Fmu}--\eqref{eq:nianxing BC}, again employing the function $w$. Consider a set with fixed $\varepsilon>0$:
    \begin{equation*}
		{\Sigma}_{-}^{\beta}\doteq\{\bm{\eta}\in\mathbb{R}^{3}: P^{\hat{\bm{\eta}}}\in \mathcal{D}_2, \text{ and } {w}^{\bm{\eta}}<1+\varepsilon~\text{on the Dirichlet boundary}~\Gamma_{cone}^{\infty}\cup\Gamma_{cone}^{\beta}\},
	\end{equation*}
    with
    \begin{equation*}
    \mathcal{D}_2\doteq\{(\xi_1,\xi_2)~|~\frac{\pi}{2}+\beta<\arctan(\xi_2/\xi_1)<\pi\}.  
    \end{equation*}
As in the above analysis for the exact solution ${w}^{\bm{\eta}}$, when $\bm{\eta}\in {\Sigma}_{-}^{\beta}$, one has $\rmD w^{\bm{\eta}}\cdot\bm{\nu}_{w}<0$ and $\rmD w^{\bm{\eta}}\cdot\bm{\nu}_{sy}<0$ on $\Gamma_{wing}\cup\Gamma_{sym}$. Moreover, from \eqref{eq phi to w}, \eqref{sup for w} and \eqref{eq lip for w}--\eqref{eq lip for phi2}, together with the assumption $\phi\geq \sqrt{1+|\bm{\xi}|^2}+\varepsilon$ in $\overline{\Omega}\setminus\overline{\Gamma_{cone}^{\infty}\cup\Gamma_{cone}^{\beta}}$, the operator $\mathcal{N}_\mu$ is locally uniformly elliptic about the solution $w$. 

By virtue of Lemma \ref{lemma: comparison principle}, one can deduce that for any $\bm{\eta}\in {\Sigma}_{-}^{\beta}$, the function $w^{\bm{\eta}}$ serves as a sub-solution to the problem \eqref{eq for w}--\eqref{BC for w} within the subdomain satisfying $w^{\bm{\eta}} > 1$; otherwise, the relation $w^{\bm{\eta}} \leq 1<w$ always holds. Accordingly, the inequality $w\geq w^{\bm{\eta}}$ is valid throughout $\overline{\Omega}\setminus\overline{\Gamma_{cone}^{\infty}\cup\Gamma_{cone}^{\beta}}$.

Let us define by ${w}^-$ the supremum of all such $w^{\bm{\eta}}$. It follows immediately that $w \geq w^-$ in $\overline{\Omega}$. Furthermore, for any subdomain separated from the Dirichlet boundary $\Gamma_{cone}^{\infty}\cup\Gamma_{cone}^{\beta}$, one can find a constant $\delta>0$ determined only by $\Omega$ so that in this subdomain
	\begin{equation}\label{uniform lowerbound}
		{w}^{-}\geq 1+\varepsilon+\delta.
	\end{equation}
In fact, for an arbitrary point $\bm{\xi}_{0}$ in the subdomain, taking $0<\delta\ll 1$ sufficiently small and setting $\bm{\eta}_{0}=(1+\varepsilon+\delta)\frac{(\bm{\xi}_{0},1)}{\sqrt{1+|\bm{\xi}_{0}|^{2}}}$ yields $\bm{\eta}_{0}\in \Sigma^\beta_{-}$, from which \eqref{uniform lowerbound} follows.

According to \eqref{uniform lowerbound} and $w \geq w^-$ in $\overline{\Omega}$, with \eqref{sup for w}, we obtain the estimate \eqref{eq phi inf}. The proof of this lemma is completed.
\end{proof}

\begin{remark}\rm\label{remark: property}
Any point on the Neumann boundaries $\Gamma_{wing}$ and $\Gamma_{sym}$ can be treated as an interior point of $\Omega$ because the problem \eqref{eq: for Fmu}--\eqref{eq:nianxing BC} is invariant under rotational transformations and possesses reflection symmetry about the axis. \qed
\end{remark}

\begin{remark}\rm \label{remark: property for epsilon}
    In the above analysis of Lemma $\ref{lemma: lip}$, to ensure the local uniform ellipticity of $\mathcal{N}_\mu$, the condition $\phi_{\mu,\varepsilon}\geq \sqrt{1+|\bm{\xi}|^2}+\varepsilon$ in $\overline{\Omega}\setminus\overline{\Gamma_{cone}^{\infty}\cup\Gamma_{cone}^{\beta}}$ can be relaxed to $\phi_{\mu,\varepsilon}\geq \sqrt{1+|\bm{\xi}|^2}$ therein and $\phi_{\mu,\varepsilon}\geq \sqrt{1+|\bm{\xi}|^2}+\varepsilon$ in its strict interior. \qed 
\end{remark}

\subsection{The proof of main theorem}
After obtaining the above Lipschitz estimate, we are in a position to prove Theorem \ref{thm4} by the strategy given at the beginning  of this section. For a fixed $\varepsilon>0$, we first show the existence of a viscosity solution
to the problem \eqref{D2conicalpotential} with \eqref{eq:nianxing BC} as follows.

\textit{Step} 1. $J_{\varepsilon}$ contains $0$. With the choice of $\mu=0$, equation \eqref{eq: for Fmu} reduces to
\begin{equation}\label{eqmueq0}
	\sum^{2}_{i,j=1} A_{ij}(0, \phi) \partial_{ij}\phi\doteq\Delta\phi+\rmD^{2}\phi[\bm{\xi},\bm{\xi}]=0\quad\text{in $\Omega$}.
\end{equation}
The problem \eqref{eqmueq0} with \eqref{eq:nianxing BC} is a mixed boundary-value problem for a linear uniformly elliptic equation. Using \cite[Theorem 1]{Lieb88} and Remark \ref{remark: property}, we obtain a unique solution $\phi_{0,\varepsilon}\in C^{1}(\overline{\Omega})\cap C^2(\Omega\cup\Gamma_{sym}\cup\Gamma_{wing})$. Besides, $\phi_{0,\varepsilon}\geq \sqrt{1+|\bm{\xi}|^2}+\varepsilon$ in $\overline{\Omega}\setminus\overline{\Gamma_{cone}^{\infty}\cup\Gamma_{cone}^{\beta}}$ follows directly from the Hopf lemma and maximum principle, and therefore $0\in J_{\varepsilon}$.

\textit{Step} 2. The closedness of $J_{\varepsilon}$. Consider a sequence $\{\mu_m\}_{m=1}^\infty \subset J_{\varepsilon}$ converging to $\mu_\infty$. Let $\phi_{m, \varepsilon}$ denote the associated solutions corresponding to $\mu_m$. To prove that the limit belongs to the set, i.e., $\mu_\infty \in J_\varepsilon$, the key point here is to establish uniform estimates for the sequence.

The uniform ellipticity of the linearized equation,
\begin{equation*}
	\sum^{2}_{i,j=1} A_{ij}(\mu_{m}, \phi_{m,\varepsilon}) \partial_{ij}\phi=0,
\end{equation*}
is guaranteed by the Lipschitz estimate derived in \eqref{eq phi inf} and \eqref{eq for phi lip}. We proceed by analyzing the estimates of $\phi_{m,\varepsilon}$ in three distinct regions: the interior, the Dirichlet boundary $\Gamma_{cone}^{\infty}\cup\Gamma_{cone}^{\beta}$, and the corner point $O$.

First, for any subdomain $\Omega_{sub} \subset \Omega\cup\Gamma_{wing}\cup\Gamma_{sym}$ located away from the corners, by Remark \ref{remark: property}, standard interior Schauder estimates (refer to \cite{GT01}) imply that the norm $\|\phi_{m,\varepsilon}\|_{C^{2,\alpha}(\Omega_{sub})}$ remains uniformly bounded with respect to $\mu_m$.

Second, regarding the boundary $\Gamma_{cone}^{\infty}\cup\Gamma_{cone}^{\beta}$ of class $C^{1,1}$, we note that the intersection corner points where this boundary meets $\Gamma_{sym}$ and $\Gamma_{wing}$ can be handled via reflection principles, effectively treating them as boundary points. Thus, by applying \cite[Theorem 5.1]{GH80}, we derive uniform $C^{1,\kappa_1}$ estimates in a neighborhood of this boundary, where $\kappa_1 \in (0,1)$.

Finally, the estimate near the corner point $O$ requires the weighted norm described in \cite{Lieb88}. Utilizing the weighted norm
\begin{align*}
    |\phi|^{-1-\kappa_2}_{2} \doteq \sup_{\delta>0}\delta^{1-\kappa_2}\|\phi\|_{C^{1,1}(\Omega_\delta)},
\end{align*}
we invoke \cite[Lemma 1.3]{Lieb88} to obtain a uniform bound $|\phi_{m,\varepsilon}|^{-1-\kappa_2}_{2} \leq C$, where $\kappa_2 \in (0,1)$, and $\Omega_\delta=\{x:~x+y\in\Omega ~\text{if}~|y|\leq\delta\}$ for every $\delta>0$. This is crucial as it ensures $\phi_{m,\varepsilon} \in C^{1,\kappa_2}$ in a neighborhood of $O$.

By the above estimates, the sequence $\{\phi_{m,\varepsilon}\}$ is uniformly bounded and equicontinuous in the space $C^1(\overline{\Omega})$. Application of the Arzel\`a-Ascoli theorem yields a subsequence $\{\phi_{m_k, \varepsilon}\}$ converging to a limit function $\phi_{\infty, \varepsilon}$ in $C^1(\overline{\Omega})$. Furthermore, the uniform interior $C^2$ estimates and the convergence imply that $\phi_{\infty,\varepsilon}\in C^1(\overline{\Omega})\cap C^2(\Omega\cup\Gamma_{sym}\cup\Gamma_{wing})$. Then, the limit $\phi_{\infty, \varepsilon}$ satisfies the boundary value problem for $\mu_\infty$ and the requisite inequality constraint, and therefore $\mu_\infty \in J_\varepsilon$.

\textit{Step} 3. The openness of $J_{\varepsilon}$. Let us fix $\mu_{0}\in J_{\varepsilon}$ and denote the corresponding solution by $\phi_{\mu_{0},\varepsilon}$. Direct linearization of equation \eqref{eq: for Fmu} at $\phi_{\mu_{0},\varepsilon}$ yields a non-negative zero-order coefficient, which poses difficulties in establishing the uniqueness of solutions to the linearized equation. To overcome this difficulty, we introduce another auxiliary function $\psi_{\mu,\varepsilon}$ as follows:
\begin{equation}\label{eq: for tran. z}
    \psi_{\mu,\varepsilon}\doteq\text{arccosh}\left(\frac{{\phi}_{\mu,\varepsilon}}{\sqrt{1+|\bm{\xi}|^{2}}}\right).
\end{equation}
Substituting \eqref{eq: for tran. z} into \eqref{eq: for Fmu}, we derive the governing equation for $\psi_{\mu,\varepsilon}$:
\begin{multline}\label{eq: for auxiliary z}
    \big(1+m(\rmD \psi)\big)(\Delta \psi+\rmD^{2}\psi[\bm{\xi},\bm{\xi}])-\mu(1+|\bm{\xi}|^{2})\rmD^{2}\psi[\rmD \psi+(\rmD \psi\cdot\bm{\xi})\bm{\xi}, \rmD \psi+(\rmD \psi\cdot\bm{\xi})\bm{\xi}]\\
    +2\big(1+(1-\mu)m(\rmD \psi)\big)\rmD \psi\cdot\bm{\xi}+\big(1+m(\rmD \psi)\big)\frac{2+(1-\mu)m(\rmD \psi)}{(1+|\bm{\xi}|^{2})\tanh \psi}=0,
\end{multline}
which can be rewritten in the operator form
\begin{equation}\label{eq1:3.1}
 \sum^{2}_{i,j=1} a_{ij}(\mu, \rmD \psi)\partial_{ij}\psi+L(\mu, \psi, \rmD \psi)=0,
\end{equation}
where the term $m(\rmD \psi)$ is given by
\begin{equation*}
	m(\rmD \psi)\doteq(1+|\bm{\xi}|^2)(|\rmD \psi|^2+|\rmD \psi\cdot\bm{\xi}|^2).
\end{equation*}
The corresponding boundary conditions of $\psi_{\mu,\varepsilon}$ can be derived from 
\eqref{eq:nianxing BC} and \eqref{eq: for tran. z}.

Note that the principal coefficients of \eqref{eq: for auxiliary z} depend solely on $\rmD \psi$, rather than on $\psi$ itself, and $\partial_{\psi}L(\mu, \psi, \rmD \psi)<0$ follows from the fact that the function $\tanh \psi$ is a strictly increasing function. These properties ensure that the zero-order coefficient of the linearized equation corresponding to equation \eqref{eq: for auxiliary z} is negative.    

The linearization of \eqref{eq1:3.1} and the corresponding boundary conditions at $\psi_{\mu_{0},\varepsilon}=\text{arccosh}({\phi}_{\mu_0,\varepsilon}/{\sqrt{1+|\bm{\xi}|^{2}}})$ is given by
\begin{equation}\label{eq for w6}
	 \sum^{2}_{i,j=1} a_{ij}(\mu_0, \rmD \psi_{\mu_{0},\varepsilon})\partial_{ij}\psi+b_i(\mu_0, \rmD \psi_{\mu_{0},\varepsilon})\partial_i \psi+\partial_{\psi}L(\mu_{0},\psi_{\mu_{0},\varepsilon}, \rmD \psi_{\mu_{0},\varepsilon})\psi=g,
\end{equation}
and
\begin{equation*}
    	\begin{cases}
		\psi=\text{arccosh}(1+\varepsilon)\quad&\text{on $\Gamma_{cone}^{\infty}\cup\Gamma_{cone}^{\beta}$},\\
		\rmD \psi\cdot\bm{\nu}_{w}=0\quad&\text{on $\Gamma_{wing}$},\\
		\rmD \psi\cdot\bm{\nu}_{sy}=0\quad&\text{on $\Gamma_{sym}$},
	\end{cases}
\end{equation*}
where
\begin{align*}
b_i(\mu_0, \rmD \psi_{\mu_{0},\varepsilon})=\sum^{2}_{i,j=1}\partial_{p_i}a_{ij}(\mu_{0}, \rmD \psi_{\mu_{0},\varepsilon})\partial_{ij}\psi_{\mu_{0},\varepsilon}+\partial_{p_i}L(\mu_{0}, \psi_{\mu_{0},\varepsilon}, \rmD \psi_{\mu_{0},\varepsilon}),
\end{align*}
for $i=1,2$, with $(p_1,p_2)\doteq \rmD \psi$. Then, by a similar discussion as in \cite[p.11 and Appendix]{CQ19}, and noting that equation \eqref{eq for w6} has the same structure as equation $(3.12)_1$ in \cite{CQ19}, we obtain that the above mixed boundary value problem admits a one-to-one mapping from $g\in C^{-1,\kappa}(\Omega)$ to $\psi\in C^{1,\kappa}(\Omega)$ satisfying
\begin{equation*}
\|\psi\|_{C^{1,\kappa}(\Omega)}\leq C\|g\|_{C^{-1,\kappa}(\Omega)}.
\end{equation*}
Here, the space $C^{-1,\kappa}(\Omega)$, as utilized in \cite{Chensx08,CQ12SCM,CQ19}, is defined as a set of functions $g$ that can be decomposed as
\begin{equation*}
g=\sum^{2}_{i=1}\partial_{\xi_i} g_i+g_0, \quad \text{with } g_i\in C^{\kappa}(\Omega) \text{ for } i=0,1,2,
\end{equation*}
equipped with the norm 
\begin{equation*}
\|g\|_{C^{-1,\kappa}(\Omega)}=\mathrm{inf}\{\sum^{2}_{i=0}\|g_i\|_{C^{\kappa}(\Omega)}\}.
\end{equation*}
This result implies the invertibility of the linearized operator of the nonlinear mapping $\mu\mapsto \psi$ at $(\mu_{0},\psi_{\mu_{0},\varepsilon})$. From the implicit function theorem, the solution can be extended to a neighborhood of $\mu_{0}$, confirming that $\mu_{0}$ is an interior point of $J_{\varepsilon}$.

Consequently, one has $J_{\varepsilon}=[0,1]$, which ensures the existence of $\phi_{1,\varepsilon}$ for any $\varepsilon > 0$.

We finally consider the limit $\mathop{\mathrm{lim}}\limits_{\varepsilon \to 0^+} \phi_{1,\varepsilon}$. The uniform estimates \eqref{eq phi inf} and \eqref{eq for phi lip} imply that the family $\{\phi_{1,\varepsilon}\}$ is uniformly bounded in $\mathrm{Lip}(\overline{\Omega})$. Then, there exists a subsequence converging in $C^0(\overline{\Omega})$ to a limit function $\phi$. Besides, the independence of the estimates \eqref{eq lip for w} and \eqref{uniform lowerbound} with respect to $\mu$ and $\varepsilon$ guarantees that equation \eqref{D2conicalpotential} is locally uniformly elliptic, i.e., uniformly elliptic in any subdomain at a positive distance from the degenerate boundary $\Gamma_{cone}^{\infty}\cup\Gamma_{cone}^{\beta}$. Thus, as in \textit{Step} 2, utilizing the interior estimate (cf. \cite[Theorem 6.17]{GT01}) and the corner estimate (cf. \cite[Lemma 1.3]{Lieb88}), we deduce that the limit function satisfies $\phi\in C^{0}(\overline{\Omega})\cap C^{1,\kappa}(\overline{\Omega}\setminus\overline{\Gamma_{cone}^{\infty}\cup\Gamma_{cone}^{\beta}})\cap C^\infty(\Omega\cup\Gamma_{sym}\cup\Gamma_{wing})$, which is a solution to Problem B. Moreover, from Remark \ref{remark: property for epsilon}, the result of Lemma \ref{lemma: lip} can be extended to the case $\varepsilon=0$, which ensures the Lipschitz boundedness of the limit $\phi$ in a neighborhood of $\overline{\Gamma_{cone}^{\infty}\cup\Gamma_{cone}^{\beta}}$, thereby establishing the regularity that $\phi\in \mathrm{Lip}(\overline{\Omega}) \cap C^{1,\kappa}(\overline{\Omega}\setminus\overline{\Gamma_{cone}^{\infty} \cup \Gamma_{cone}^{\beta}})\cap C^{\infty}(\Omega\cup\Gamma_{sym} \cup \Gamma_{wing})$. 

The proof of the main theorem is completed. 

\section{Further discussions}\label{sec:Further}
In this section, we will present a brief discussion on the possible global flow field structures of the $\Lambda$-wing proposed by K\"uchemann in \cite{KD65}. We will also analyze the case of an asymmetric conical wing with $\Lambda$-shaped cross sections, in which the left and right wings have different sweep angles (see \eqref{asy sweep wing} below). This case can be regarded as the direction of the velocity of incoming flow no longer parallel to the $x_1Ox_3$-plane in Problem A above.

\subsection{On the possible flow field structure of the \texorpdfstring{$\Lambda$}--wing}\label{sec:possible}
Based on the preliminary numerical calculations, for the conical wing with $\Lambda$-shaped cross sections, K\"uchemann put forward a speculation on the global flow field structure at different anhedral angles $\beta$; see Figure \ref{fgpossible}, which is redrawn from \cite[p.304]{KD65}. For Chaplygin gas, we have previously verified the flow field structure $(b)$ and $(c)$ in Figure \ref{fgpossible}; namely, the case $(b)$ corresponds to the critical angle $\beta=\beta_c$ (see Figure \ref{fgbetaeqc}), and the case $(c)$ corresponds to $\beta=0$ (see Figure \ref{figbeta0}). However, we find a new flow field structure (i.e., $\beta=\beta_0$, see Figure \ref{fgbetaeqbeta0}) in which there exists an oblique shock perpendicular to the conical wing, quite different from the case $(a)$. Thus, it is natural to ask whether the flow field structure of the case $(a)$ exists.

%-----------------------------------fig-----------
\begin{figure}[htb]
\centering
\includegraphics[scale=0.45]{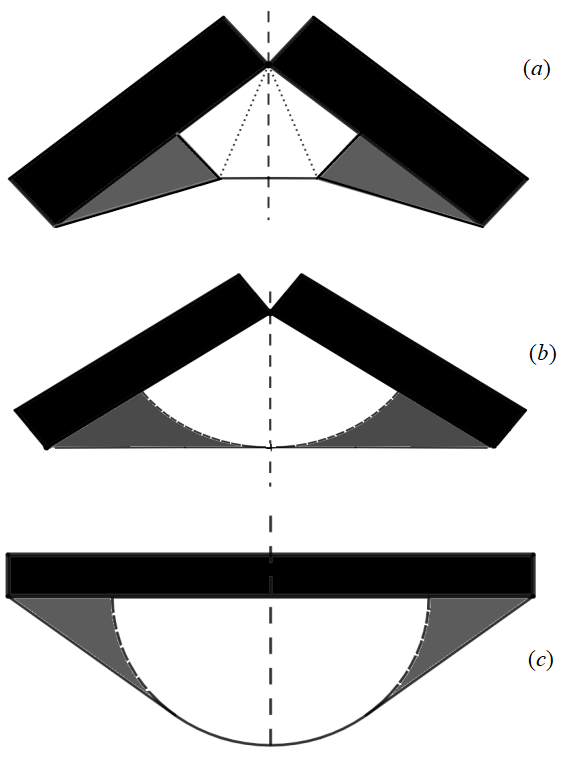}
\caption{Possible mixed flow regimes.}\label{fgpossible}
\end{figure}
 %-----------------------------------fig-----------

For the Chaplygin gas, if the case $(a)$ of Figure \ref{fgpossible} exists, then since any shock is characteristic, the straight shock that connects to the oblique shock attached to the leading edge of the wing must be tangent to the characteristic cone $\mathcal{C}_\infty$ corresponding to the incoming flow. Also, as analyzed in Section \ref{sec2.2}, this situation only occurs at the anhedral angles $\beta=\beta_c$, which is the case $(b)$ of Figure \ref{fgpossible}. Thus, we can confirm that the case $(a)$ cannot exist.

As for the polytropic gas, we only give a preliminary analysis. Owing to the structural symmetry, it is sufficient to consider the half-side for the case $(a)$ of Figure \ref{fgpossible}, in which there exists a three-shock pattern. From the construction of Mach reflection in \cite{supshockCF}, we know that to maintain the stability of the three-shock pattern, it is necessary to introduce a contact discontinuity. Additionally, the flow on both sides of the contact discontinuity is parallel to it. This implies that once the contact discontinuity exists, it must be parallel to the surface of the wing as the flow satisfies the slip boundary condition, unless the flow is along the $x_3$-axis. To explain the location of the degenerate curves, we present a simplified diagram (see Figure \ref{simplify}), which corresponds to the left half of the case $(a)$. 

 %-----------------------------------fig-----------
\begin{figure}[htb]
\centering
\includegraphics[scale=1]{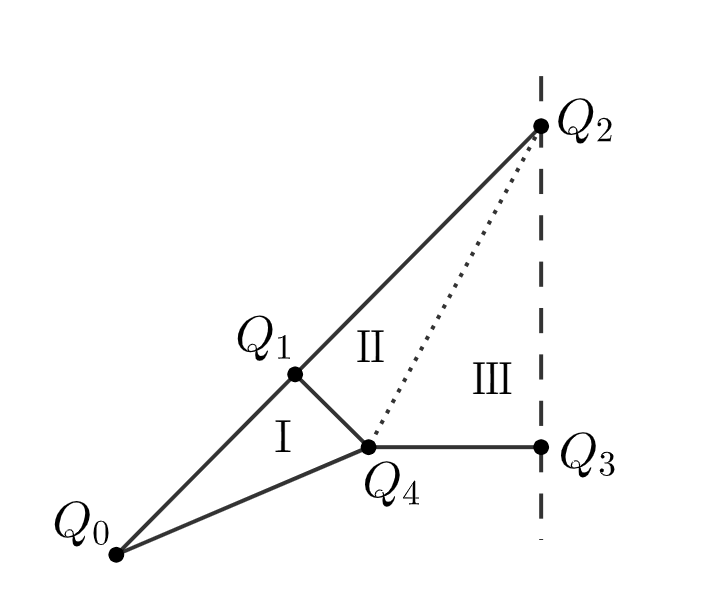}
\caption{The possible structure for conical wing.}\label{simplify}
\end{figure}
 %-----------------------------------fig-----------

In Figure \ref{simplify}, the line $Q_{0}Q_{2}$ denotes the compression surface of the conical wing; the lines $Q_0Q_4$, $Q_1Q_4$ and $Q_3Q_4$ denote the oblique shock waves; and $Q_2Q_4$ denotes the contact discontinuity; the symbols I, II, III stand for the domains in which the flow is uniform determined by the oblique shocks $Q_0Q_4$, $Q_1Q_4$ and $Q_3Q_4$. Note that the flow in the domains II and III must be along the $x_3$-axis; otherwise, the contact discontinuity $Q_2Q_4$ will be parallel to the compression surface of the conical wing. 

It is feasible to construct a particular solution that satisfies the above-mentioned structure by the shock polars for the full compressible Euler equations. However, verifying the stability of the above flow field structure is non-trivial; in other words, determining whether the above structure will occur is not straightforward. We leave this for future work.

\subsection{Conical wings with asymmetric cross-sections}\label{sec:Asywing}
In the previous sections, we always assume that the conical wing is symmetrical about the $x_{1}Ox_{3}$-plane. So it is natural to consider an asymmetric one $W^\beta_{\sigma,\hat{\sigma}}$ defined by \eqref{asy sweep wing} below, which can also be regarded as an incoming flow with a velocity component along the $x_2$-axis past the conical wing defined by \eqref{sweep wing}. By abuse of notation, we still adopt the previous notations.

%-----------------------------------fig-----------
\begin{figure}[htbp]
\centering
\begin{minipage}[t]{0.9\textwidth}
\centering
\subfigure[$\beta=0$]{
\includegraphics[width=5cm]{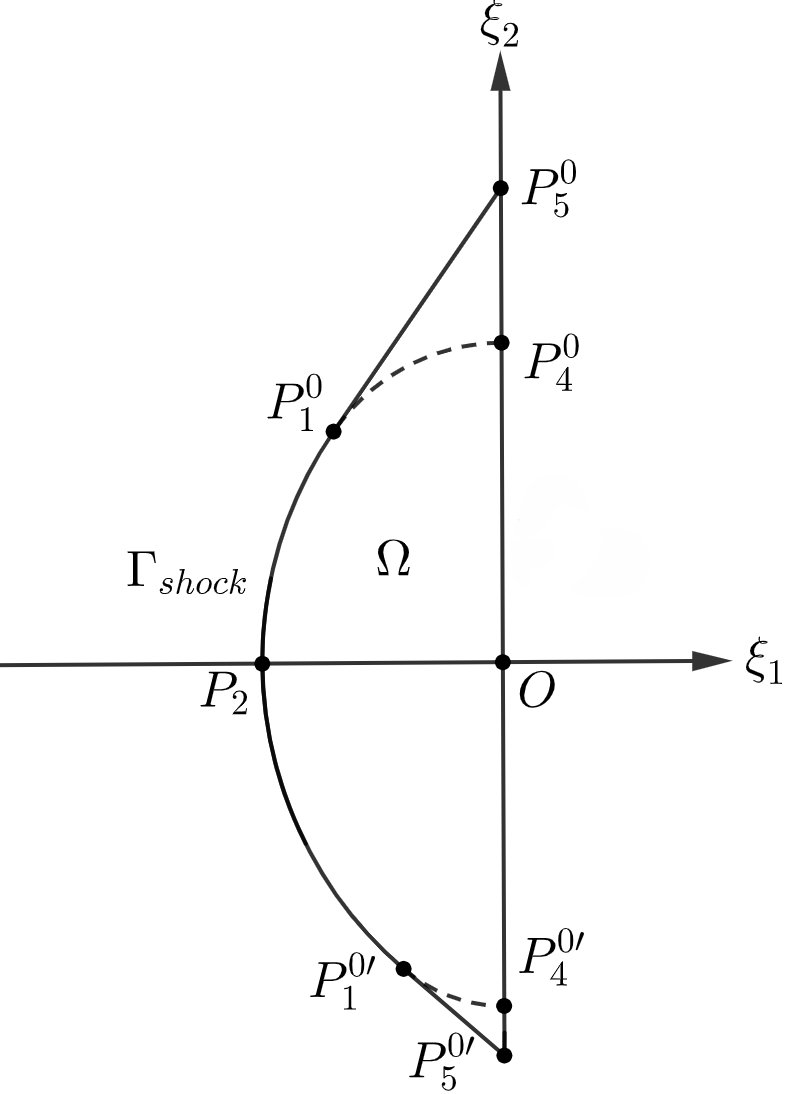}
\label{1}}
\subfigure[$0<\beta<\beta_c$]{
\includegraphics[width=5cm]{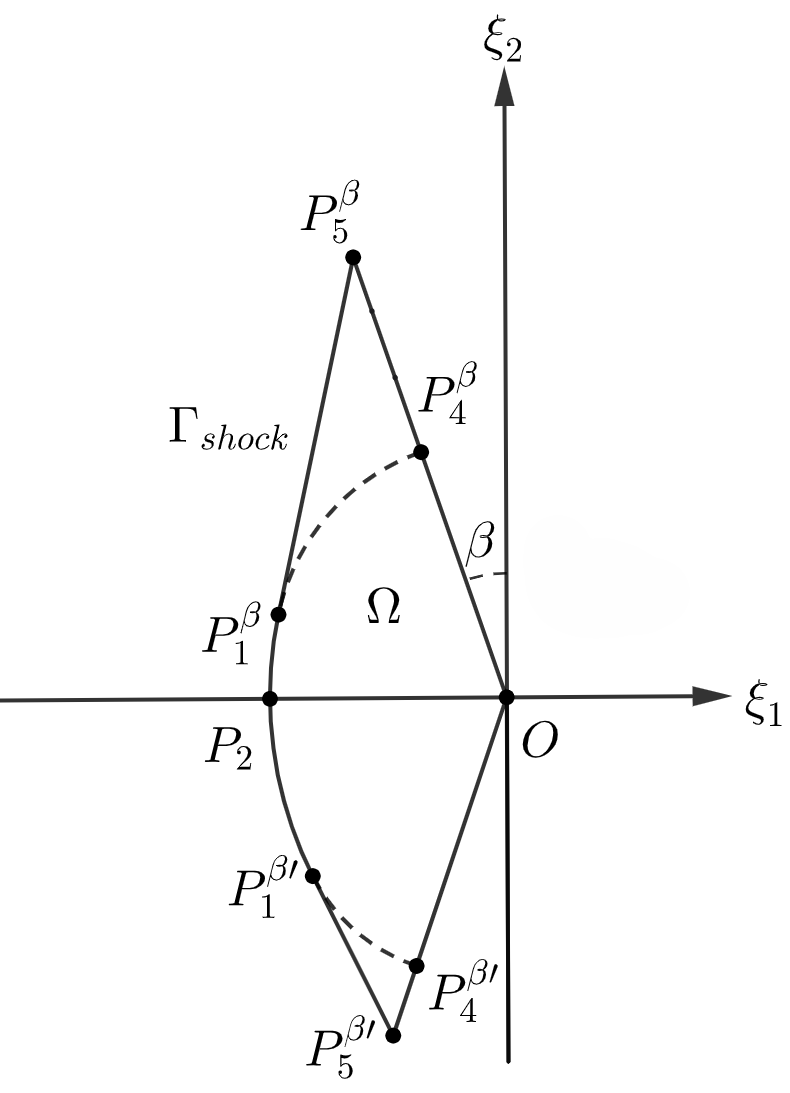}
\label{2}}
\end{minipage}
\caption{Patterns of shock waves for asymmetric wing.}
\end{figure}
 %-----------------------------------fig-----------

Let us begin with an asymmetric conical wing $W^\beta_{\sigma,\hat{\sigma}}$ with $\Lambda$-shaped cross section, given by
\begin{equation} \label{asy sweep wing}
	W^\beta_{\sigma,\hat{\sigma}}=\{(x_1,x_2,x_3):x_1=-|x_2|\tan\beta,-x_{3}\cot\hat{\sigma}\cos\beta<x_{2}<x_{3}\cot\sigma\cos\beta,x_{3}>0\},
\end{equation}
where $\sigma,\hat{\sigma},\beta\in(0,{\pi}/{2})$. 
Without loss of generality, we assume that $\sigma<\hat{\sigma}$. For any fixed $\alpha\in(0,\alpha_0)$ and $\sigma,\hat{\sigma}\in(0,\sigma_0]$, using what we did in Section \ref{sec2pre}, we can derive the location of the shock and the uniform flow state outside the Mach cone when $\beta$ is less than a critical angle, where one of the attached shocks is perpendicular to the wing $W^\beta_{\sigma,\hat{\sigma}}$. In the following, we briefly present the structures of shock waves and necessary explanations, but omit the detailed proof herein.

 %-----------------------------------fig-----------
\begin{figure}[htbp]
\centering
\begin{minipage}[t]{0.9\textwidth}
\centering
\subfigure[$\beta=\beta_c$]{
\includegraphics[width=5cm]{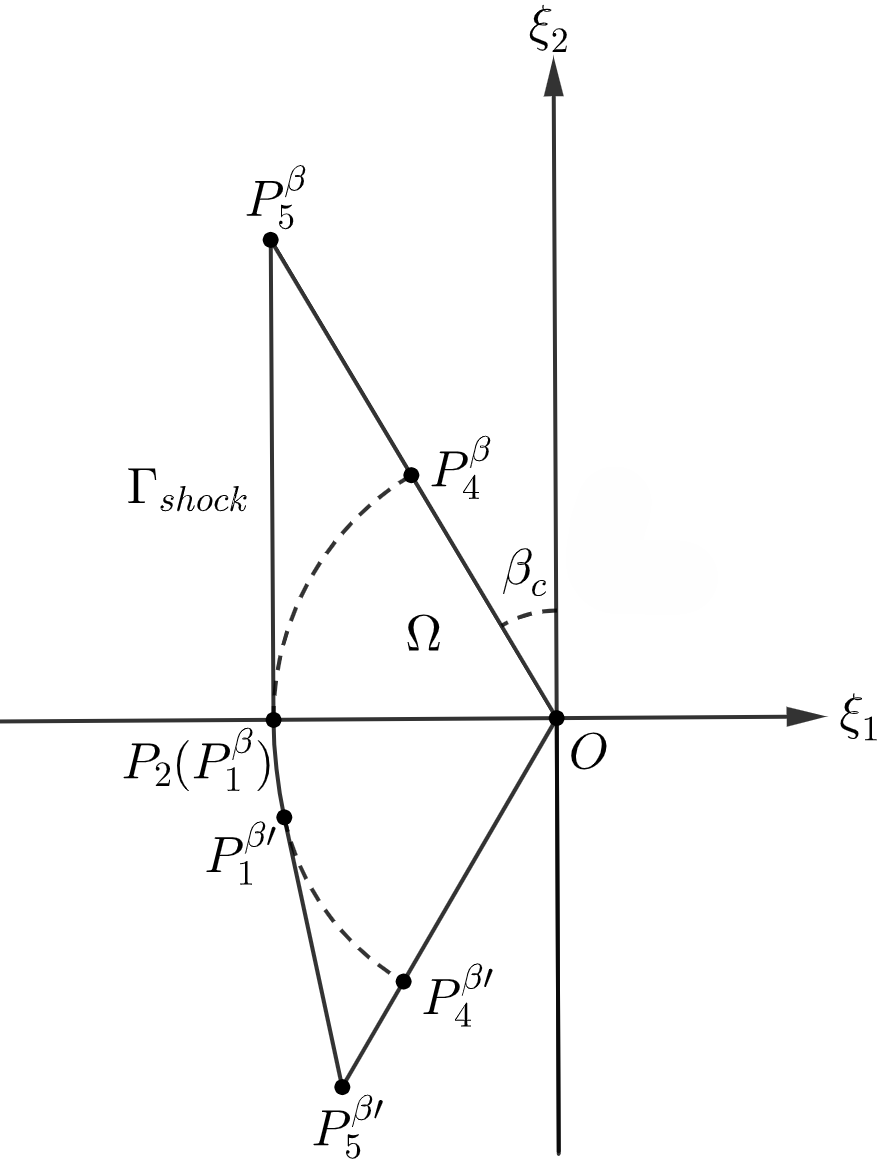}
\label{3}}
\subfigure[$\beta=\beta^a_c$]{
\includegraphics[width=5cm]{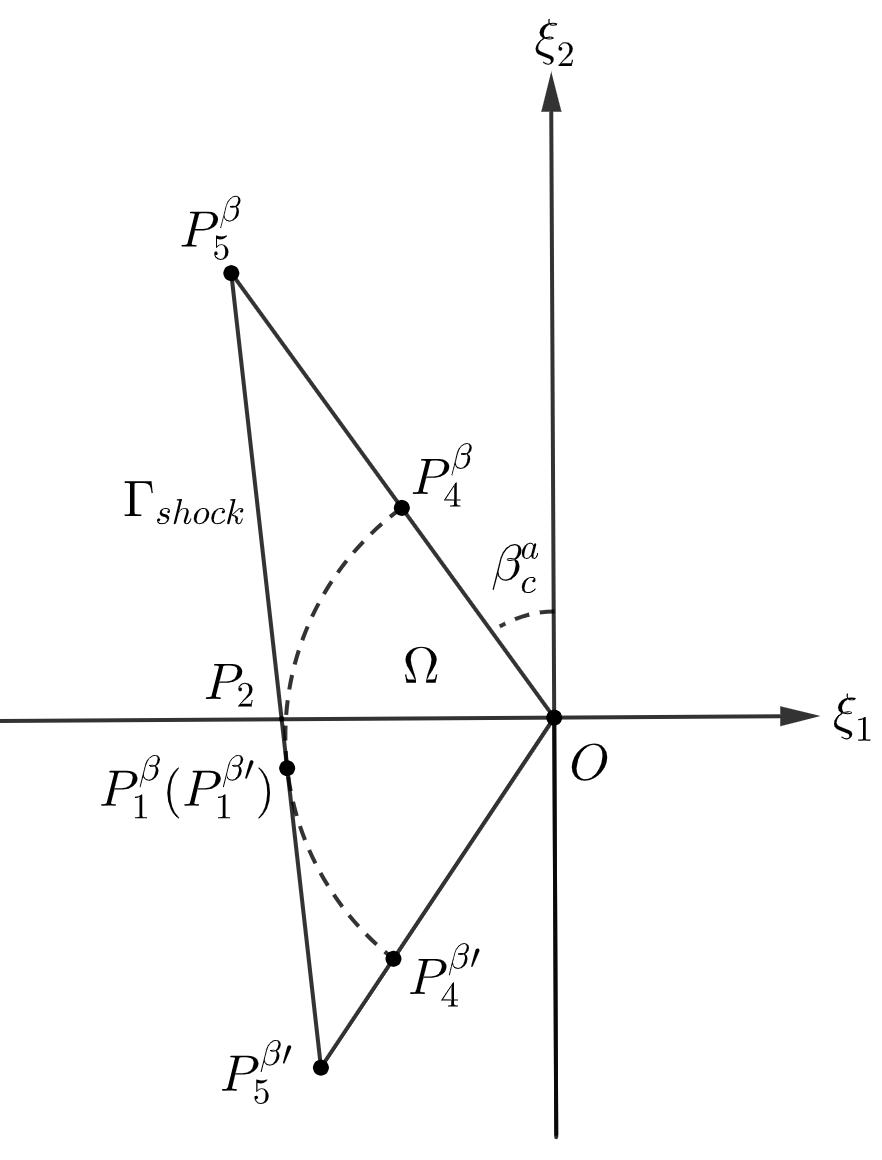}
\label{4}}
\end{minipage}
\caption{Patterns of shock waves for asymmetric wing.}
\end{figure}
 %-----------------------------------fig-----------

 %-----------------------------------fig-----------
\begin{figure}[htbp]
\centering
\begin{minipage}[t]{0.9\textwidth}
\centering
\subfigure[$\beta^a_c<\beta<\beta^a_0$]{
\includegraphics[width=4.7cm]{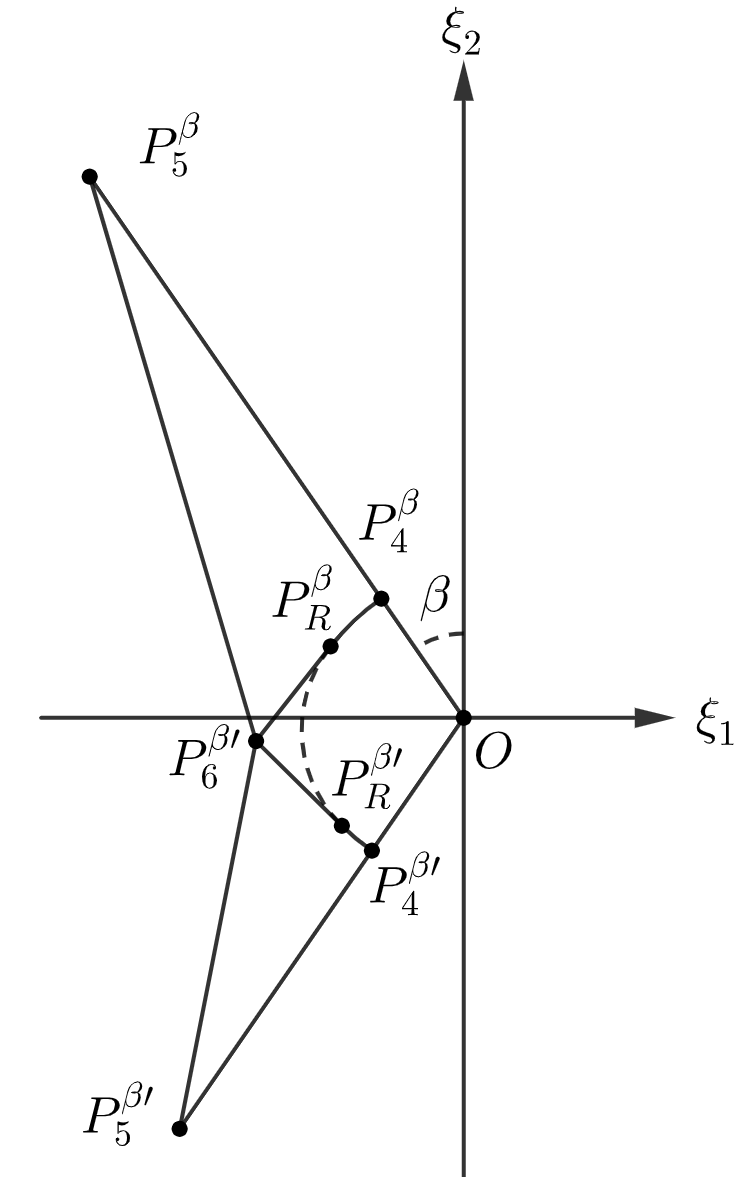}
\label{5}}
\subfigure[$\beta=\beta^a_0$]{
\includegraphics[width=4.7cm]{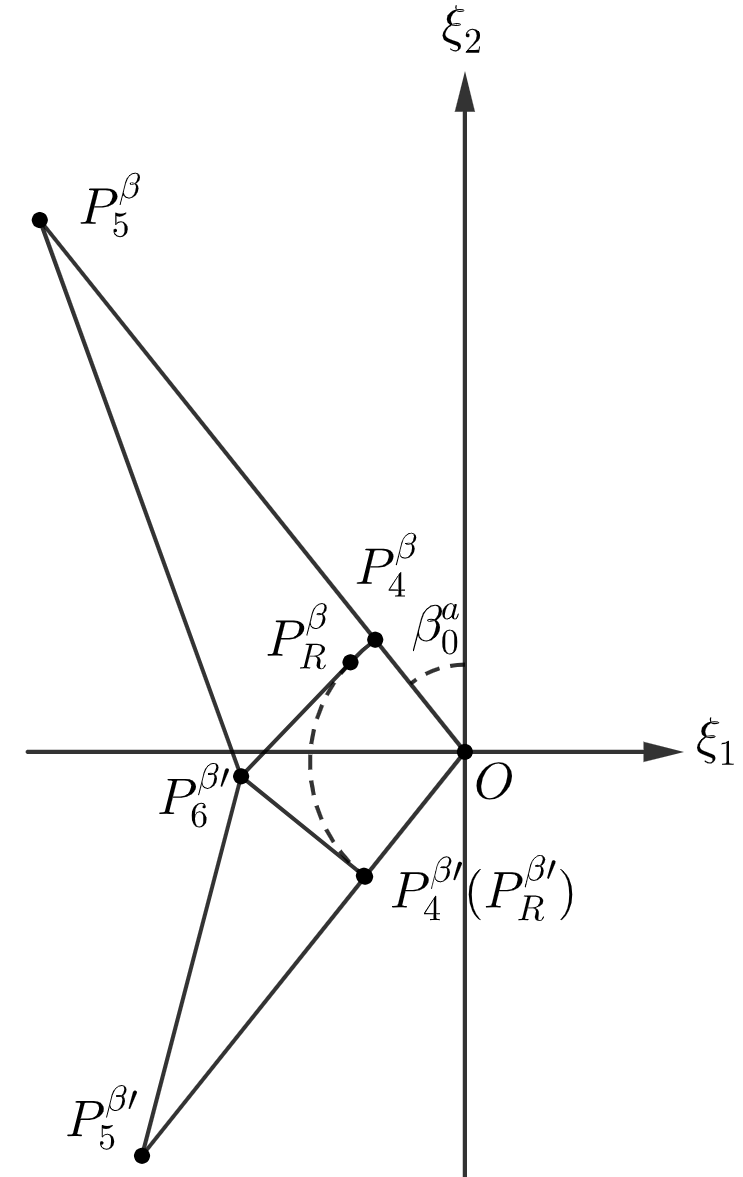}
\label{6}}
\end{minipage}
\caption{Patterns of shock waves for asymmetric wing.}
\end{figure}
 %-----------------------------------fig-----------
 
As discussed in \cite[Section 2.2]{LY22}, we can obtain the global flow field structure in Figure \ref{1} for the case $\beta=0$. Besides, under the assumption that $\sigma<\hat{\sigma}$, it is easy to verify that when $0<\beta<\beta_c$, neither of the two attached shock waves is perpendicular to the $\xi_1$-axis (see Figure \ref{2}), and when $\beta=\beta_c$, one of them is perpendicular to the $\xi_1$-axis (see Figure \ref{3}), where $\beta_c$ is defined by \eqref{eqdefinebetac}. Moreover, there must exist a critical angle $\beta^a_c>\beta_c$ such that the attached shock becomes planar, which corresponds to the flow field structure of Nonweiler wing (see Figure \ref{4}). Furthermore, there also exists a critical angle $\beta^a_0$ such that when $\beta^a_c<\beta<\beta^a_0$, two attached shocks intersect at a point and generate two oblique resulting shocks that do not reach the wing $W^\beta_{\sigma,\hat{\sigma}}$ (see Figure \ref{5}); when $\beta=\beta^a_0$, one of the oblique resulting shocks reaches and is perpendicular to the wing $W^\beta_{\sigma,\hat{\sigma}}$ (see Figure \ref{6}). Notably, further analysis is required to determine which of the resulting shocks will first become perpendicular to the wing $W^\beta_{\sigma,\hat{\sigma}}$.

As for the flow inside the Mach cone, the corresponding problems can be investigated by the method developed in Section \ref{sec:3} with more refined analytical techniques. We also leave this for future works.

\appendix
\section{Shock polar for Chaplygin gas}\label{appendix a}
From the properties introduced in Section \ref{sec1}, we present the basic results about the shock polar for the Chaplygin gas here, and the detailed calculations can be found in \cite[Appendix A]{LY22}. 

For a given state of incoming flow $(\rho_0,(u_0, 0))$ with $u_0 > c_0$ and $c_0 =a/\rho_0> 0$, the shock line $\mathcal{S}$ is thereby determined with $\overrightarrow{OO_0}=(u_0,0)$ and $|O_0P|=c_0$; then the angle $\gamma$ between $\mathcal{S}$ and the direction of the incoming flow satisfies $\sin\gamma=\frac{c_{0}}{u_{0}}$. Moreover, the state of downstream flow $(\rho_1,(u_1, v_1))$ can be uniquely determined by the direction of its velocity $\theta=\arctan(\frac{v_1}{u_1})$ with $\overrightarrow{OO_1}=(u_1,v_1)$ and $|O_1P|=c_1 $, i.e.,
\begin{align*}
    c_1=\frac{c_0-\tan\theta\sqrt{u^2_0-c^2_0}}{\sqrt{u^2_0-c^2_0}+c_0\tan\theta}\sqrt{u^2_0-c^2_0}
\end{align*}
and
\begin{align}\label{u1v1polar}
u_1=\frac{u_0\sqrt{u^2_0-c^2_0}}{\sqrt{u^2_0-c^2_0}+c_0\tan\theta}, \quad\quad
v_1=\frac{u_0\tan\theta\sqrt{u^2_0-c^2_0}}{\sqrt{u^2_0-c^2_0}+c_0\tan\theta}.
\end{align}
Thus, the trajectory of the point $O_1(u_{1},v_{1})$ in the $(u, v)$-plane describes the shock polar $O_0P$ as $\theta$ varies, where the coordinates of $O_0$ are  $(u_{0},0)$; $P$ is the tangent point between the shock line $\mathcal{S}$ and the circle $C_{1}$ with center $O_{1}(u_{1},v_{1})$ and radius $c_{1}=a/\rho_{1}$ (see Figure \ref{fgshockpolar}).

\begin{figure}[H]
	\centering
	\begin{tikzpicture}[scale = 1.1, smooth]
	\draw [-latex] (-1,0)--(5,0) node[right] {\footnotesize$u$};
	\draw [-latex] (0,-2.2)--(0,2.5) node[above] {\footnotesize$v$};
	\node [below left, font=\footnotesize] (0,0) {$O$};
	\draw [thick,-stealth] (0,0)--(2.5,0) node[below=-1pt, font=\footnotesize] {$O_0$};
	\draw (2.5,0) circle (1.5);
	\draw (0,0)--(1.6,1.2)--(2.8,2.1) node[right, font=\footnotesize] {$\mathcal{S}$};
	\draw (1.6,1.2)--(2.5,0);
	\draw (1.52,1.14)--(1.58,1.06)--(1.66,1.12);
	\draw [thick,-stealth] (0,0)--(1.9,0.8) node[below=1pt, font=\footnotesize] {$O_1$};
	\draw [densely dotted](1.9,0.8) circle (0.51);
	\draw (0:0.3) arc (0:22:0.3);
	\node [right, font=\footnotesize] at (0.21,0.08) {\tiny$\theta$};
	\draw (0:0.6) arc (0:-8:0.6);
	\draw (0:0.5) arc (0:38:0.5);
	\node [right, font=\footnotesize] at (0.4,0.15) {\tiny$\gamma$};
	\fill(2.5,0)circle(1.3pt);
	\fill(1.9,0.8)circle(1.3pt);
	\fill(1.6,1.2)circle(1.3pt);
	\node [left, font=\footnotesize] at (1.6,1.2) {$P$};
	\draw [-latex](2.25,0.35)--(3.0,0.35) node[right] {\footnotesize{shock polar}};
	\end{tikzpicture}
	\caption{Shock polar for a Chaplygin gas.}
	\label{fgshockpolar}
\end{figure}
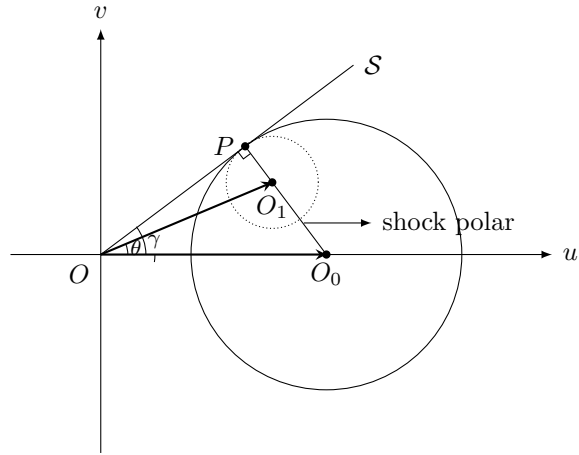

It follows from Figure \ref{fgshockpolar} that as $\theta \to \gamma$, the flow behind the shock will concentrate. To avoid this phenomenon, the incoming flow should satisfy the following condition
\begin{align}\label{restrictpolar}
    c_0<u_0<\frac{c_0}{\sin\theta},
\end{align}
which is deduced from the relations $\theta<\gamma$ and $\sin\gamma=\frac{c_{0}}{u_{0}}$.

\section*{Acknowledgments} 
The authors are grateful to Prof. Myoungjean Bae (KAIST) for her comments and suggestions on the manuscript. This study was supported by the Natural Science Foundation of Hubei under Grant No. 2024AFB007, and the Science and Technology Commission of Shanghai Municipality under Grants No.\,24ZR1420000 and No.\,22DZ2229014.

\medskip

\noindent{\bf Data Availability:} The paper does not use any data set. 

\medskip 

\noindent \textbf{Declarations Conflict of Interest:}  The authors state that there is no conflict of interest.

%\bibliographystyle{plain}

%\bibliography{references}
\end{document}